\documentclass[12pt,lenq]{amsart}
\usepackage{amsmath}
\usepackage{amscd}
\usepackage{amssymb}
\usepackage{amsbsy}
\usepackage{amsfonts}
\usepackage{latexsym}
\usepackage{graphics}
\usepackage{graphicx}
\usepackage{amsmath,amscd,latexsym}
\usepackage{multirow}
\usepackage{array}
\usepackage{paralist}
\usepackage{titletoc}

\theoremstyle{plain}
\newtheorem{theorem}{Theorem}[section]
\newtheorem{proposition}[theorem]{Proposition}
\newtheorem{lemma}[theorem]{Lemma}
\newtheorem{corollary}[theorem]{Corollary}
\newtheorem{remark}[theorem]{Remark}
\newtheorem{definition}[theorem]{Definition}
\newtheorem{notation}[theorem]{Notation}
\newtheorem{example}[theorem]{Example}
\newtheorem{main theorem}[theorem]{Main Theorem}

\newtheorem{question}[theorem]{Question}
\newtheorem{convention}[theorem]{Convention}

\newcommand{\interior}{\operatorname{int}}

\newcommand{\ZZ}{\mathbb{Z}}
\newcommand{\QQ}{\mathbb{Q}}
\newcommand{\RR}{\mathbb{R}}

\newcommand{\HH}{\mathbb{H}}
\newcommand{\QQQ}{\hat{\mathbb{Q}}}
\newcommand{\RRR}{\hat{\mathbb{R}}}

\newcommand{\Conway}{\mbox{\boldmath$S$}^{2}}

\newcommand{\PConway}{\mbox{\boldmath$S$}}

\newcommand{\rtangle}[1]{(B^3,t({#1}))}

\newcommand{\DD}{\mathcal{D}}
\newcommand{\RGPC}[2]{\Gamma({#1};{#2})}
\newcommand{\RGPP}[1]{\hat\Gamma_{#1}}
\newcommand{\RGP}[1]{\Gamma_{#1}}

\newcommand{\Hecke}{\mbox{$G$}}

\newcommand{\orbs}{\mbox{\boldmath$S$}}

\newcommand{\svert}{\,|\,}

\newcommand{\llangle}{\langle\langle}
\newcommand{\rrangle}{\rangle\rangle}

\newcommand{\lp}{(\hskip -0.07cm (}
\newcommand{\rp}{)\hskip -0.07cm )}

\makeatletter
\renewcommand\subsection{\@startsection{subsection}{2}{0mm}
    {-10.5dd plus-8pt minus-4pt}{10.5dd}
     {\normalsize\upshape}}
\makeatother

\begin{document}

\title{Homotopically equivalent simple loops
on 2-bridge spheres in Heckoid orbifolds for 2-bridge links (I)}

\author{Donghi Lee}
\address{Department of Mathematics\\
Pusan National University \\
San-30 Jangjeon-Dong, Geumjung-Gu, Pusan, 609-735, Republic of Korea}
\email{donghi@pusan.ac.kr}

\author{Makoto Sakuma}
\address{Department of Mathematics\\
Graduate School of Science\\
Hiroshima University\\
Higashi-Hiroshima, 739-8526, Japan}
\email{sakuma@math.sci.hiroshima-u.ac.jp}

\subjclass[2010]{Primary 20F06, 57M25\\
\indent {The first author was supported by Basic Science Research Program
through the National Research Foundation of Korea(NRF) funded
by the Ministry of Education, Science and Technology(2012R1A1A3009996).
The second author was supported
by JSPS Grants-in-Aid 22340013.}}

\begin{abstract}
In this paper and its sequel,
we give a necessary and sufficient condition
for two essential simple loops on a $2$-bridge sphere
in an even Heckoid orbifold for a $2$-bridge link
to be homotopic in the orbifold.
We also give a necessary and sufficient condition
for an essential simple loop on a $2$-bridge sphere
in an even Heckoid orbifold for a $2$-bridge link
to be peripheral or torsion in the orbifold.
This paper treats the case when
the $2$-bridge link is a $(2,p)$-torus link,
and its sequel will treat the remaining cases.
\end{abstract}
\maketitle

%\tableofcontents

\section{Introduction}
Let $K(r)$ be the $2$-bridge link of slope $r \in \QQ$
and let $n$ be an integer or a half-integer greater than $1$.
In \cite{lee_sakuma_6}, following Riley's work~\cite{Riley2},
we introduced the {\it Heckoid group $\Hecke(r;n)$ of index $n$ for $K(r)$}
as the orbifold fundamental group
of the {\it Heckoid orbifold $\orbs(r;n)$ of index $n$ for $K(r)$}.
The classical Hecke groups introduced in \cite{Hecke}
are essentially the simplest Heckoid groups.
According to whether $n$ is an integer or a non-integral half-integer,
the Heckoid group $\Hecke(r;n)$ and the Heckoid orbifold $\orbs(r;n)$
are said to be {\it even} or {\it odd}.
The even Heckoid orbifold $\orbs(r;n)$
is the $3$-orbifold satisfying the following conditions (see Figure~\ref{fig.Hekoid-orbifold}):
\begin{enumerate}[\rm (i)]
\item
The underlying space $|\orbs(r;n)|$ is the exterior,
$E(K(r))=S^3-\interior N(K(r))$, of $K(r)$.
\item
The singular set is the lower tunnel of $K(r)$,
where the index of the singularity is $n$.
\end{enumerate}
For a description of odd Heckoid orbifolds,
see \cite[Proposition~5.3]{lee_sakuma_6}.

\begin{figure}[h]
\begin{center}
\includegraphics{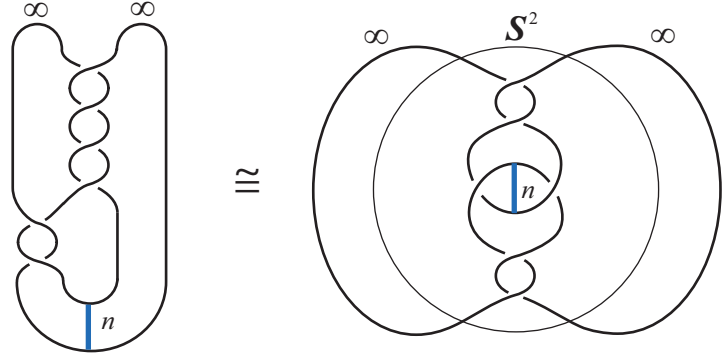}
\end{center}
\caption{
\label{fig.Hekoid-orbifold}
The even Heckoid orbifold $\orbs(r;n)$
of index $n$ for the $2$-bridge link $K(r)$.
Here $(S^3,K(r))=(B^3,t(\infty))\cup (B^3,t(r))$ is the $2$-bridge link
with $r=2/9=[4,2]$ (with a single component).
The rational tangles $(B^3,t(\infty))$ and $(B^3,t(r))$, respectively,
are the outside and the inside of the bridge sphere $\Conway$.
The underlying space of the orbifold is the complement
of an open regular neighborhood of
the subgraph consisting of those edges with weight $\infty$.
The singular set of the orbifold is the edge with weight $n$,
with cone angle $2\pi/n$.
}
\end{figure}

In \cite[Theorem~2.3]{lee_sakuma_6},
we gave a systematic construction
of upper-meridian-pair-preserving epimorphisms
from $2$-bridge link groups onto
Heckoid groups,
generalizing Riley's construction in \cite{Riley2}.
Furthermore, we proved, in \cite[Theorem~2.4]{lee_sakuma_7}, that
all upper-meridian-pair-preserving epimorphisms
from $2$-bridge link groups onto {\it even} Heckoid groups
are contained in those constructed in \cite[Theorem~2.3]{lee_sakuma_6}.
To prove this result, we determined
those essential simple loops on a $2$-bridge sphere
in an even Heckoid orbifold $\orbs(r;n)$
which are null-homotopic in $\orbs(r;n)$ (see \cite[Theorem~2.3]{lee_sakuma_7}).

The purpose of this paper and its sequel~\cite{lee_sakuma_10} is (i) to give
a necessary and sufficient condition
for two essential simple loops on a $2$-bridge sphere
in an even Heckoid orbifold $\orbs(r;n)$
to be homotopic in $\orbs(r;n)$,
and (ii) to give a necessary and sufficient condition
for an essential simple loop on a $2$-bridge sphere
in an even Heckoid orbifold $\orbs(r;n)$
to be peripheral or torsion in $\orbs(r;n)$.
In this paper, we treat the case when
$K(r)$ is a $(2,p)$-torus link,
and the sequel~\cite{lee_sakuma_10} will treat the remaining cases.
These results will be used, in our upcoming work,
to show the existence of a variation of McShane's identity
for even Heckoid orbifolds.
For an overview of this series of works, we refer the reader to
the research announcement~\cite{lee_sakuma_8}.

The remainder of this paper is organized as follows.
In Section~\ref{statements},
we describe the main results.
In Section~\ref{group_presentation}, we introduce
the upper presentation of an even Heckoid group,
and review basic facts concerning its single relator established in \cite{lee_sakuma}.
In Section~\ref{sec:small_cancellation_theory},
we apply small cancellation theory to conjugacy diagrams over
the upper presentations of even Heckoid groups.
In Section~\ref{sec:technical_lemmas},
we establish technical lemmas which will play essential roles in the succeeding sections.
Finally, Sections~\ref{sec:proof of main theorem(1) for the case when $r=1/m$}--\ref{sec:proof of main theorem(2)}
are devoted to the proof of Main Theorem~\ref{thm:conjugacy}.

\section{Main results}
\label{statements}

We quickly recall notation and basic facts introduced in \cite{lee_sakuma_6}.
The {\it Conway sphere} $\PConway$ is the 4-times punctured sphere
which is obtained as the quotient of $\RR^2-\ZZ^2$
by the group generated by the $\pi$-rotations around
the points in $\ZZ^2$.
For each $s \in \QQQ:=\QQ\cup\{\infty\}$,
let $\alpha_s$ be the simple loop in $\PConway$
obtained as the projection of a line in $\RR^2-\ZZ^2$
of slope $s$.
We call $s$ the {\it slope} of the simple loop $\alpha_s$.

For each $r\in \QQQ$,
the {\it $2$-bridge link $K(r)$ of slope $r$}
is the sum of the rational tangle
$\rtangle{\infty}$ of slope $\infty$ and
the rational tangle $\rtangle{r}$ of slope $r$.
Recall that $\partial(B^3-t(\infty))$ and $\partial(B^3-t(r))$
are identified with $\PConway$
so that $\alpha_{\infty}$ and $\alpha_r$
bound disks in $B^3-t(\infty)$ and $B^3-t(r)$, respectively.
By van-Kampen's theorem, the link group $G(K(r))=\pi_1(S^3-K(r))$ is obtained as follows:
\[
G(K(r))=\pi_1(S^3-K(r))
\cong \pi_1(\PConway)/ \llangle\alpha_{\infty},\alpha_r\rrangle
\cong \pi_1(B^3-t(\infty))/\llangle\alpha_r\rrangle.
\]
We call the image in $G(K(r))$ of the meridian pair of
$\pi_1(B^3-t(\infty))$ the {\it upper meridian pair}
(see \cite[Figure~3]{lee_sakuma}).

On the other hand, if $r$ is a rational number and $n\ge 2$ is an integer,
then by the description of the even Heckoid orbifold $\orbs(r;n)$ in the introduction,
the even Hekoid group $\Hecke(r;n)$,
which is defined as the orbifold fundamental group of $\orbs(r;n)$, is identified with
\[
\Hecke(r;n)
\cong\pi_1(\PConway)/ \llangle\alpha_{\infty},\alpha_r^n\rrangle
\cong \pi_1(B^3-t(\infty))/\llangle\alpha_r^n\rrangle.
\]
In particular,
the even Heckoid group $\Hecke(r;n)$ is a two-generator and one-relator group.
We also call the image in $\Hecke(r;n)$
of the meridian pair of $\pi_1(B^3-t(\infty))$
the {\it upper meridian pair}.

We are interested in the following naturally arising question.

\begin{question}
\label{question1}
{\rm
For $r$ a rational number and $n$ an integer or a half-integer
greater than $1$,
consider the Heckoid orbifold $\orbs(r;n)$ of index $n$ for the $2$-bridge link $K(r)$.
\begin{enumerate}[\indent \rm (1)]
\item
Which essential simple loop $\alpha_s$ on $\PConway$
is null-homotopic in $\orbs(r;n)$?
\item
For two distinct essential simple loops $\alpha_s$ and $\alpha_{s'}$ on $\PConway$,
when are they homotopic in $\orbs(r;n)$?
\item
Which essential simple loop $\alpha_s$ on $\PConway$
is peripheral or torsion in $\orbs(r;n)$?
\end{enumerate}
}
\end{question}

This is an analogy of a natural question for $2$-bridge links,
which has the origin in Minsky's question~\cite[Question~5.4]{Gordon},
and which was completely solved in the series of papers
\cite{lee_sakuma, lee_sakuma_2, lee_sakuma_3, lee_sakuma_4}
and applied in \cite{lee_sakuma_5}.
See \cite{lee_sakuma_0} for an overview of these works and
\cite{Ohshika-Sakuma} for a related work.

We note that
(1) a loop in the orbifold $\orbs(r;n)$ is {\it null-homotopic} in $\orbs(r;n)$
if and only if it determines the trivial conjugacy class of the Heckoid group $\Hecke(r;n)$, and
(2) two loops in $\orbs(r;n)$ are {\it homotopic}  in $\orbs(r;n)$
if and only if they determine the same conjugacy class in $\Hecke(r;n)$
(see \cite{BMP, Boileau-Porti} for the concept of homotopy in orbifolds).
We say that a loop in $\orbs(r;n)$ is {\it peripheral} if and only if it is homotopic to a loop
in the paring annulus naturally associated with $\orbs(r;n)$ (see \cite [Section~6]{lee_sakuma_6}),
i.e., it represents the conjugacy class of a power of a meridian of $\Hecke(r;n)$.
We also say that a loop in $\orbs(r;n)$ is {\it torsion} if it represents the conjugacy class
of a non-trivial torsion element of $\Hecke(r;n)$.
If we identify $\Hecke(r;n)$ with a Kleinian group generated by two parabolic transformations
(see \cite[Theorem~2.2]{lee_sakuma_6}),
then a loop $\orbs(r;n)$ is peripheral or torsion
if and only if it corresponds to a parabolic transformation or a non-trivial elliptic transformation
accordingly.
Thus Question~\ref{question1} can be interpreted as a question on the Heckoid group $\Hecke(r;n)$.

Let $\DD$ be the {\it Farey tessellation}
of the upper half plane $\HH^2$.
Then $\QQQ$ is identified with the set of the ideal vertices of $\DD$.
Let $\RGP{\infty}$ be the group of automorphisms of
$\DD$ generated by reflections in the edges of $\DD$ with an endpoint $\infty$.
For $r$ a rational number
and $n$ an integer or a half-integer greater than $1$,
let $C_r(2n)$ be the group of automorphisms of $\DD$ generated
by the parabolic transformation, centered on the vertex $r$,
by $2n$ units in the clockwise direction,
and let $\RGPC{r}{n}$ be the group generated by $\RGP{\infty}$ and $C_r(2n)$.
Suppose that $r$ is not an integer, i.e., $K(r)$ is not a trivial knot.
Then $\RGPC{r}{n}$ is the free product $\RGP{\infty}*C_r(2n)$
having a fundamental domain, $R$, shown in Figure~\ref{fig.fd_orbifold}.
Here, $R$ is obtained as the intersection of fundamental domains
for $\RGP{\infty}$ and $C_r(2n)$, and so
$R$ is bounded by the following two pairs of Farey edges:
\begin{enumerate}[\indent \rm (i)]
\item
the pair of adjacent Farey edges with an endpoint $\infty$
which cut off a region in $\bar\HH^2$ containing $r$, and
\item
a pair of Farey edges with an endpoint $r$
which cut off a region in $\bar\HH^2$ containing $\infty$ such that
one edge is the image of the other by a generator of $C_r(2n)$.
\end{enumerate}

Let $\bar{I}(r;n)$ be the union of
two closed intervals
in $\partial \HH^2=\RRR$
obtained as the intersection of the closure of $R$ and $\partial \HH^2$.
(In the special case when $r\equiv \pm1/p \pmod{\ZZ}$
for some integer $p\ge 2$,
one of the intervals may be degenerated to a single point.)
Note that there is a pair $\{r_1, r_2\}$
of boundary points of $\bar{I}(r;n)$
such that $r_2$ is the image of $r_1$ by a generator of $C_r(2n)$.
Set $I(r;n):=\bar{I}(r;n) -\{r_i\}$ with $i=1$ or $2$.
Note that
$I(r;n)$ is the disjoint union of a closed interval and a half-open interval,
except possibly for the special case when $r\equiv \pm 1/p \pmod{\ZZ}$.
Even in the exceptional case, we can choose $R$ so that
$I(r;n)$ satisfies this condition.

\begin{figure}[htbp]
\begin{center}
\includegraphics{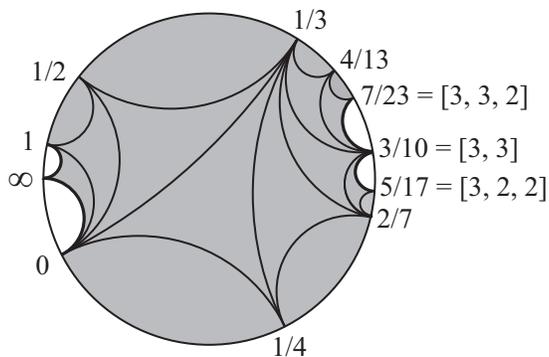}
\end{center}
\caption{\label{fig.fd_orbifold}
A fundamental domain of $\RGPC{r}{n}$ in the
Farey tessellation (the shaded domain) for $r=3/10=${\scriptsize $\cfrac{1}{3+\cfrac{1}{3}}$}\,$=:[3,3]$ and $n=2$.
In this case, $\bar{I}(r;n)=[0,5/17] \cup [7/23,1]$.}
\end{figure}

As a sufficient condition for each of Question~\ref{question1}(1) and (2),
we were able to obtain the following (cf. \cite[Theorem~2.4]{lee_sakuma_6}).

\begin{proposition}[{\cite[Theorem~2.2]{lee_sakuma_7}}]
\label{prop:fundametal_domain}
Suppose that $r$ is a non-integral rational number and
that $n$ is an integer or a half-integer greater than $1$.
Then, for any $s\in\QQQ$, there is a unique rational number
$s_0\in I(r;n) \cup \{\infty, r\}$
such that $s$ is contained in the $\RGPC{r}{n}$-orbit of $s_0$.
Moreover, the conjugacy classes
$\alpha_s$ and $\alpha_{s_0}$ in $\Hecke(r;n)$ are equal.
In particular, if $s_0=\infty$, then $\alpha_s$ is the trivial conjugacy class
in $\Hecke(r;n)$.
\end{proposition}

Furthermore, for even Heckoid groups, we proved that
the converse of the last assertion of Proposition~\ref{prop:fundametal_domain}
is also true, implying that
the sufficient condition for Question~\ref{question1}(1) is actually a necessary and sufficient condition.
This made us possible to describe all upper-meridian-pair-preserving
epimorphisms from $2$-bridge link groups onto even Heckoid groups.

\begin{proposition}[{\cite[Theorem~2.3]{lee_sakuma_7}}]
\label{prop:reformulation}
Suppose that $r$ is a non-integral rational number and
that $n$ is an integer greater than $1$.
Then $\alpha_s$ represents the trivial element of $\Hecke(r;n)$
if and only if $s$ belongs to the $\RGPC{r}{n}$-orbit of $\infty$.
In other words, if $s\in I(r;n) \cup \{r\}$, then
$\alpha_s$ does not represent the trivial element of $\Hecke(r;n)$.
\end{proposition}

The purpose of the present paper and its sequel~\cite{lee_sakuma_10}
is to give the following complete solution to each of Question~\ref{question1}(2) and (3)
for even Heckoid orbifolds..

\begin{main theorem}
\label{thm:conjugacy}
Suppose that $r$ is a non-integral rational number and
that $n$ is an integer greater than $1$.
Then the following hold.
\begin{enumerate}[\indent \rm (1)]
\item The simple loops $\{\alpha_s \svert s\in I(r;n)\}$ represent
mutually distinct conjugacy classes in $\Hecke(r;n)$.

\item There is no rational number $s \in I(r;n)$
for which $\alpha_s$ is peripheral in $\Hecke(r;n)$.

\item There is no rational number $s \in I(r;n)$
for which $\alpha_s$ is torsion in $\Hecke(r;n)$.
\end{enumerate}
\end{main theorem}

Note that (1) together with (3) implies that the simple loops
$\{\alpha_s \svert s\in I(r;n) \cup \{r\}\}$
are not mutually homotopic in $\orbs(r;n)$,
because $\alpha_{r}$ is a nontrivial torsion element in $\Hecke(r;n)$.
Thus, together with Proposition~\ref{prop:reformulation},
the above theorem gives a complete answer to Question~\ref{question1}
for even Heckoid orbifolds.

In this paper, we give a proof of Main Theorem~\ref{thm:conjugacy}
when $K(r)$ is a torus link,
i.e., $r\equiv \pm1/p \pmod{1}$ for some integer $p\ge 2$.
As in \cite{lee_sakuma_7}, the key tool used in the proofs is small cancellation theory,
applied to two-generator and one-relator presentations,
so-called the upper presentations, of even Heckoid groups.

At the end of this section, we point out the following fact,
which can be easily proved (cf. \cite[Lemma~4.1]{lee_sakuma_6}).
By the lemma, we may assume $0< r\le 1/2$.

\begin{lemma}
\label{lemma:natural-homeomprphisms}
For any rational number $r$ and an integer $n\ge 2$,
the following hold.
\begin{enumerate}[\indent \rm (1)]
\item
There is an {\rm (}orientation-preserving{\rm )} orbifold-homeomorphism $f$
from $\orbs(r;n)$ to $\orbs(r+1;n)$
which maps the $2$-bridge sphere of $\orbs(r;n)$ to
that of $\orbs(r+1;n)$.
Moreover, the restriction of $f$ to the $2$-bridge sphere
maps the simple loop $\alpha_s$ to $\alpha_{s+1}$
for any $s\in\QQQ$.
\item
There is an {\rm (}orientation-preserving{\rm )} orbifold-homeomorphism $f$
from $\orbs(r;n)$ to $\orbs(-r;n)$
which maps the $2$-bridge sphere of $\orbs(r;n)$ to
that of $\orbs(-r;n)$.
Moreover, the restriction of $f$ to the $2$-bridge sphere
maps the simple loop $\alpha_s$ to $\alpha_{-s}$
for any $s\in\QQQ$.
\end{enumerate}
\end{lemma}

\section{Upper presentations of even Heckoid groups and review of basic facts from \cite{lee_sakuma}}
\label{group_presentation}

In this section, we introduce the
upper presentation of an even Heckoid group $\Hecke(r;n)$,
where $r$ is a rational number and $n \ge 2$ is an integer,
and review basic facts established in \cite{lee_sakuma}
concerning it.
These facts are to be used throughout this paper and its sequel~\cite{lee_sakuma_8}.

In order to describe the upper presentations of even Heckoid groups,
recall that
\[
\Hecke(r;n)\cong
\pi_1(\PConway)/ \llangle\alpha_{\infty},\alpha_r^n\rrangle
\cong
\pi_1(B^3-t(\infty))/ \llangle\alpha_r^n\rrangle.
\]
Let $\{a,b\}$ be the standard meridian generator pair of $\pi_1(B^3-t(\infty), x_0)$
as described in \cite[Section~3]{lee_sakuma}
(see also \cite[Section~5]{lee_sakuma_0}).
Then $\pi_1(B^3-t(\infty))$ is identified with the free group $F(a,b)$.
Obtain a word $u_r\in F(a,b)\cong\pi_1(B^3-t(\infty))$
which is represented by the simple loop $\alpha_r$.
It then follows that
\[
\Hecke(r;n) \cong\pi_1(B^3-t(\infty))/\llangle \alpha_r^n\rrangle
\cong \langle a, b \svert u_r^n \rangle.
\]
This two-generator and one-relator presentation
is called the {\it upper presentation} of the even Heckoid group $\Hecke(r;n)$,
which is used throughout the remainder of this paper.
It is known by \cite[Proposition~1]{Riley}
that there is a nice formula to find $u_r$ as follows.
(For a geometric description, see \cite[Section~5]{lee_sakuma_0}.)

\begin{lemma}
\label{presentation}
Let $p$ and $q$ be relatively prime integers
such that $p \ge 1$.
For $1 \le i \le p-1$, let
\[\epsilon_i = (-1)^{\lfloor iq/p \rfloor},\]
where $\lfloor x \rfloor$ is the greatest integer not exceeding $x$.
\begin{enumerate}[\indent \rm (1)]
\item If $p$ is odd, then
\[
u_{q/p}=a\hat{u}_{q/p}b^{(-1)^q}\hat{u}_{q/p}^{-1},\]
where
$\hat{u}_{q/p} = b^{\epsilon_1} a^{\epsilon_2} \cdots b^{\epsilon_{p-2}} a^{\epsilon_{p-1}}$.
\item If $p$ is even, then
\[
u_{q/p}=a\hat{u}_{q/p}a^{-1}\hat{u}_{q/p}^{-1},\]
where
$\hat{u}_{q/p} = b^{\epsilon_1} a^{\epsilon_2} \cdots a^{\epsilon_{p-2}} b^{\epsilon_{p-1}}$.
\end{enumerate}
\end{lemma}

\begin{remark}
\label{rem:epsilon}
{\rm
For $r=0/1, 1/1$ and $1/0$, we have $u_{0/1}=ab$, $u_{1/1}=ab^{-1}$ and $u_{1/0}=1$.
}
\end{remark}

Now we define the sequences $S(r)$ and $T(r)$ and the cyclic sequences
$CS(r)$ and $CT(r)$, all of which are read from $u_r$ defined in Lemma~\ref{presentation},
and review several important properties of these sequences
from \cite{lee_sakuma} so that we can adopt small cancellation theory.
To this end we fix some definitions and notation.
Let $X$ be a set.
By a {\it word} in $X$, we mean a finite sequence
$x_1^{\epsilon_1}x_2^{\epsilon_2}\cdots x_n^{\epsilon_n}$
where $x_i\in X$ and $\epsilon_i=\pm1$.
Here we call $x_i^{\epsilon_i}$ the {\it $i$-th letter} of the word.
For two words $u, v$ in $X$, by
$u \equiv v$ we denote the {\it visual equality} of $u$ and
$v$, meaning that if $u=x_1^{\epsilon_1} \cdots x_n^{\epsilon_n}$
and $v=y_1^{\delta_1} \cdots y_m^{\delta_m}$ ($x_i, y_j \in X$; $\epsilon_i, \delta_j=\pm 1$),
then $n=m$ and $x_i=y_i$ and $\epsilon_i=\delta_i$ for each $i=1, \dots, n$.
For example, two words $x_1x_2x_2^{-1}x_3$ and $x_1x_3$ ($x_i \in X$) are {\it not} visually equal,
though $x_1x_2x_2^{-1}x_3$ and $x_1x_3$ are equal as elements of the free group with basis $X$.
The length of a word $v$ is denoted by $|v|$.
A word $v$ in
$X$ is said to be {\it reduced} if $v$ does not contain $xx^{-1}$ or $x^{-1}x$ for any $x \in X$.
A word is said to be {\it cyclically reduced}
if all its cyclic permutations are reduced.
A {\it cyclic word} is defined to be the set of all cyclic permutations of a
cyclically reduced word. By $(v)$ we denote the cyclic word associated with a
cyclically reduced word $v$.
Also by $(u) \equiv (v)$ we mean the {\it visual equality} of two cyclic words
$(u)$ and $(v)$. In fact, $(u) \equiv (v)$ if and only if $v$ is visually a cyclic shift
of $u$.

\begin{definition}
\label{def:alternating}
{\rm (1) Let $v$ be a reduced word in
$\{a,b\}$. Decompose $v$ into
\[
v \equiv v_1 v_2 \cdots v_t,
\]
where, for each $i=1, \dots, t-1$, all letters in $v_i$ have positive (resp., negative) exponents,
and all letters in $v_{i+1}$ have negative (resp., positive) exponents.
Then the sequence of positive integers
$S(v):=(|v_1|, |v_2|, \dots, |v_t|)$ is called the {\it $S$-sequence of $v$}.

(2) Let $(v)$ be a cyclic word in
$\{a, b\}$. Decompose $(v)$ into
\[
(v) \equiv (v_1 v_2 \cdots v_t),
\]
where all letters in $v_i$ have positive (resp., negative) exponents,
and all letters in $v_{i+1}$ have negative (resp., positive) exponents (taking
subindices modulo $t$). Then the {\it cyclic} sequence of positive integers
$CS(v):=\lp |v_1|, |v_2|, \dots, |v_t| \rp$ is called
the {\it cyclic $S$-sequence of $(v)$}.
Here the double parentheses denote that the sequence is considered modulo
cyclic permutations.

(3) A reduced word $v$ in $\{a,b\}$ is said to be {\it alternating}
if $a^{\pm 1}$ and $b^{\pm 1}$ appear in $v$ alternately,
i.e., neither $a^{\pm2}$ nor $b^{\pm2}$ appears in $v$.
A cyclic word $(v)$ is said to be {\it alternating}
if all cyclic permutations of $v$ are alternating.
In the latter case, we also say that $v$ is {\it cyclically alternating}.
}
\end{definition}

\begin{definition}
\label{def4.1(3)}
{\rm
For a rational number $r$ with $0<r\le 1$,
let $u_r$ be defined as in Lemma~\ref{presentation}.
Then the symbol $S(r)$ (resp., $CS(r)$) denotes the
$S$-sequence $S(u_r)$ of $u_r$
(resp., cyclic $S$-sequence $CS(u_r)$ of $(u_r)$), which is called
the {\it S-sequence of slope $r$}
(resp., the {\it cyclic S-sequence of slope $r$}).}
\end{definition}

Throughout this paper unless specified otherwise,
we suppose that $r$ is a rational number
with $0<r\le1$
(cf. Lemma~\ref{lemma:natural-homeomprphisms}),
and write $r$ as a continued fraction:
\begin{center}
\begin{picture}(230,70)
\put(0,48){$\displaystyle{
r=[m_1,m_2, \dots,m_k]:=
\cfrac{1}{m_1+
\cfrac{1}{ \raisebox{-5pt}[0pt][0pt]{$m_2 \, + \, $}
\raisebox{-10pt}[0pt][0pt]{$\, \ddots \ $}
\raisebox{-12pt}[0pt][0pt]{$+ \, \cfrac{1}{m_k}$}
}},}$}
\end{picture}
\end{center}
where $k \ge 1$, $(m_1, \dots, m_k) \in (\mathbb{Z}_+)^k$ and
$m_k \ge 2$ unless $k=1$. For brevity, we write $m$ for $m_1$.

\begin{lemma} [{\cite[Proposition~4.3]{lee_sakuma}}]
\label{lem:properties}
The following hold.
\begin{enumerate}[\indent \rm (1)]
\item Suppose $k=1$, i.e., $r=1/m$.
Then $S(r)=(m,m)$.

\item Suppose $k\ge 2$. Then each term of $S(r)$ is either $m$ or $m+1$,
and $S(r)$ begins with $m+1$ and ends with $m$.
Moreover, the following hold.

\begin{enumerate}[\rm (a)]
\item If $m_2=1$, then no two consecutive terms of $S(r)$ can be $(m, m)$,
so there is a sequence of positive integers $(t_1,t_2,\dots,t_s)$ such that
\[
S(r)=(t_1\langle m+1\rangle, m, t_2\langle m+1\rangle, m, \dots,
t_s\langle m+1\rangle, m).
\]
Here, the symbol ``$t_i\langle m+1\rangle$'' represents $t_i$ successive $m+1$'s.

\item If $m_2 \ge 2$, then no two consecutive terms of $S(r)$ can be $(m+1, m+1)$,
so there is a sequence of positive integers $(t_1,t_2,\dots,t_s)$ such that
\[
S(r)=(m+1, t_1\langle m\rangle, m+1, t_2\langle m\rangle,
\dots,m+1, t_s\langle m\rangle).
\]
Here, the symbol ``$t_i\langle m\rangle$'' represents $t_i$ successive $m$'s.
\end{enumerate}
\end{enumerate}
\end{lemma}

\begin{definition}
\label{def_T(r)}
{\rm
If $k\ge 2$, the symbol $T(r)$ denotes the sequence
$(t_1,t_2,\dots,t_s)$ in Lemma~\ref{lem:properties},
which is called the {\it $T$-sequence of slope $r$}.
The symbol $CT(r)$ denotes the cyclic
sequence represented by $T(r)$, which is called the
{\it cyclic $T$-sequence of slope $r$}.
}
\end{definition}

\begin{example}
\label{cyclic_sequence}
{\rm (1) Let $r={10}/{37}=[3,1,2,3]$.
By Lemma~\ref{presentation}, we
see that the $S$-sequence of $\hat{u}_r$ is
\[
S(\hat{u}_r)=(3, 4, 4, 3, 4, 4, 3, 4, 4, 3).
\]
By the formula for $u_r$ in Lemma~\ref{presentation},
this implies
\[
S(r)=S(u_r) =
(\underbrace{4,4,4}_3,3,\underbrace{4, 4}_2, 3, \underbrace{4, 4}_2, 3,
\underbrace{4,4,4}_3,3, \underbrace{4, 4}_2, 3, \underbrace{4, 4}_2, 3).
\]
So $T(r)=(3, 2, 2, 3, 2, 2)$ and $CT(r) = \lp 3, 2, 2, 3, 2, 2 \rp$.

(2) Let $r={8}/{35}=[4,2,1,2]$.
Again by Lemma~\ref{presentation},
we obtain that the $S$-sequence of $\hat{u}_r$ is
\[
S(\hat{u}_r)=(4, 4, 5, 4, 4, 5, 4, 4).
\]
By the formula for $u_r$ in Lemma~\ref{presentation},
this implies
\[
S(r)=S(u_r)=
(5, \underbrace{4}_1, 5,
\underbrace{4, 4}_2, 5, \underbrace{4, 4}_2, 5, \underbrace{4}_1, 5,
\underbrace{4, 4}_2, 5, \underbrace{4, 4}_2).
\]
So $T(r) = (1, 2, 2, 1, 2, 2)$ and $CT(r) = \lp 1, 2, 2, 1, 2, 2 \rp$.
}
\end{example}

\begin{lemma} [{\cite[Proposition~4.4 and Corollary~4.6]{lee_sakuma}}]
\label{lem:induction1}
Let $\tilde{r}$ be the rational number defined as
\[
\tilde{r}=
\begin{cases}
[m_3, \dots, m_k] & \text{if $m_2=1$};\\
[m_2-1, m_3, \dots, m_k] & \text{if $m_2 \ge 2$}.
\end{cases}
\]
Then we have $CS(\tilde{r})=CT(r)$.
\end{lemma}

\begin{lemma} [{\cite[Proposition~4.5]{lee_sakuma}}]
\label{lem:sequence}
The sequence $S(r)$ has a decomposition $(S_1, S_2, S_1, S_2)$ which satisfies the following.
\begin{enumerate} [\indent \rm (1)]
\item Each $S_i$ is symmetric,
i.e., the sequence obtained from $S_i$ by reversing the order is
equal to $S_i$. {\rm (}Here, $S_1$ is empty if $k=1$.{\rm )}
\item Each $S_i$
{\rm (}if it is not empty{\rm )}
occurs only twice in the cyclic sequence $CS(r)$.
\item $S_1$
{\rm (}if it is not empty{\rm )}
begins and ends with $m+1$.
\item $S_2$ begins and ends with $m$.
\end{enumerate}
\end{lemma}

\begin{example}
{\rm (1) Let $r={10}/{37}=[3,1,2,3]$.
Recall from Example~\ref{cyclic_sequence} that
\[
S(r) =(4,4,4, 3, 4, 4, 3, 4, 4, 3, 4,4,4,3,4, 4, 3, 4, 4, 3).
\]
Putting $S_1=(4,4,4)$ and $S_2=(3, 4, 4, 3, 4, 4, 3)$, we have
\[
S(r)= (S_1, S_2, S_1, S_2),
\]
where $S_1$ and $S_2$ satisfy all the assertions in Lemma~\ref{lem:sequence}.

(2) Let $r={8}/{35}=[4,2,1,2]$. Recall also from
Example~\ref{cyclic_sequence} that
\[
S(r) = (5, 4, 5, 4, 4, 5, 4, 4, 5, 4, 5, 4, 4, 5, 4, 4).
\]
Putting $S_1=(5, 4, 5)$ and $S_2=(4, 4, 5, 4, 4)$, we also have
\[
S(r) = (S_1, S_2, S_1, S_2),
\]
where $S_1$ and $S_2$ satisfy all the assertions in Lemma~\ref{lem:sequence}.
}
\end{example}

\begin{remark}
\label{remark:recovering the slope}
{\rm
By using the fact that $u_r$ is obtained from
the line of slope $r$ in $\RR^2-\ZZ^2$
by reading its intersection with the vertical lattice lines,
we see that the slope $s=q/p$ is recovered from $CS(s)=\lp S_1,S_2,S_1,S_2 \rp$
by the rule that $p$ is the sum of the
terms of $S_1$ and $S_2$
whereas $q$ is the sum of the lengths of $S_1$ and $S_2$.
}
\end{remark}

\begin{lemma} [{\cite[Proof of Proposition~4.5]{lee_sakuma}}]
\label{lem:relation}
Let $\tilde{r}$ be the rational number defined as in Lemma~{\rm \ref{lem:induction1}}.
Also let $S(\tilde{r})=(T_1, T_2, T_1, T_2)$ and $S(r) =(S_1, S_2, S_1, S_2)$
be decompositions described as in Lemma~{\rm \ref{lem:sequence}}.
Then the following hold.
\begin{enumerate} [\indent \rm (1)]
\item If $m_2=1$ and $k=3$, then $T_1=\emptyset$, $T_2=(m_3)$,
and $S_1 =(m_3\langle m+1 \rangle)$, $S_2 =(m)$.

\item If $m_2=1$ and $k\ge 4$, then
$T_1=(t_1, \dots, t_{s_1})$, $T_2=(t_{s_1+1}, \dots, t_{s_2})$, and
\[
\begin{aligned}
S_1
&=(t_1 \langle m+1 \rangle,  m, t_2 \langle m+1 \rangle,
\dots, t_{s_1-1}\langle m+1 \rangle, m, t_{s_1}\langle m+1 \rangle), \\
S_2 &=(m, t_{s_1+1}\langle m+1\rangle,m, \dots, m, t_{s_2}\langle m+1\rangle, m).
\end{aligned}
\]

\item If $k=2$, then $T_1=\emptyset$, $T_2=(m_2-1)$,
and $S_1 =(m+1)$, $S_2 =((m_2-1)\langle m \rangle)$.

\item If $m_2 \ge 2$ and $k\ge 3$, then $T_1=(t_1, \dots, t_{s_1})$,
$T_2=(t_{s_1+1}, \dots, t_{s_2})$, and
\[
\begin{aligned}
S_1 &=(m+1, t_{s_1+1}\langle m\rangle, m+1,
\dots, m+1, t_{s_2}\langle m\rangle, m+1),\\
S_2 &=(t_1\langle m\rangle, m+1,t_2\langle m\rangle, \dots,
t_{s_1-1}\langle m\rangle, m+1, t_{s_1}\langle m\rangle).
\end{aligned}
\]
\end{enumerate}
\end{lemma}

\begin{example}
\label{example:S-sequence}
{\rm
(1) If $r=[2, 1, 5]$, then
$\tilde{r}=[5]$ by Lemma~\ref{lem:induction1}.
So by Lemma~\ref{lem:properties}(1),
$S(\tilde{r})=(\emptyset,5,\emptyset,5)$.
Thus by Lemma~\ref{lem:relation}(1),
\[
S(r)=(5 \langle 3 \rangle, 2, 5 \langle 3 \rangle, 2),
\]
where $S_1=(5 \langle 3 \rangle)$ and $S_2=(2)$.

(2) If $r=[2, 5]$, then
$\tilde{r}=[4]$ by Lemma~\ref{lem:induction1}.
So by Lemma~\ref{lem:properties}(1),
$S(\tilde{r})=(\emptyset,4,\emptyset,4)$.
Thus by Lemma~\ref{lem:relation}(3),
\[
S(r)=(3, 4 \langle 2 \rangle, 3, 4 \langle 2 \rangle),
\]
where $S_1=(3)$ and $S_2=(4 \langle 2 \rangle)$.
}
\end{example}

By Lemmas~\ref{lem:properties} and \ref{lem:relation},
we easily obtain the following corollary.

\begin{corollary}
\label{cor:cor-to-relation}
Let $S(r)=(S_1,S_2,S_1,S_2)$ be as in Lemma~{\rm \ref{lem:sequence}}.
Then the following hold.
\begin{enumerate} [\indent \rm (1)]
\item If $m_2=1$, then $(m+1,m+1)$ appears in $S_1$.
\item If $m_2\ge 2$ and if $r \ne [m,2]=2/(2m+1)$,
then $(m,m)$ appears in $S_2$.
\end{enumerate}
\end{corollary}

\section{Small cancellation theory}
\label{sec:small_cancellation_theory}

\subsection{Basic definitions and preliminary facts}

Let $F(X)$ be the free group with basis $X$. A subset $R$ of $F(X)$
is said to be {\it symmetrized},
if all elements of $R$ are cyclically reduced and, for each $w \in R$,
all cyclic permutations of $w$ and $w^{-1}$ also belong to $R$.

\begin{definition}
{\rm Suppose that $R$ is a symmetrized subset of $F(X)$.
A nonempty word $b$ is called a {\it piece} if there exist distinct $w_1, w_2 \in R$
such that $w_1 \equiv bc_1$ and $w_2 \equiv bc_2$.
The small cancellation conditions $C(p)$ and $T(q)$,
where $p$ and $q$ are integers such that $p \ge 2$ and $q \ge 3$,
are defined as follows (see \cite{lyndon_schupp}).
\begin{enumerate}[\indent \rm (1)]
\item Condition $C(p)$: If $w \in R$
is a product of $t$ pieces, then $t \ge p$.

\item Condition $T(q)$: For $w_1, \dots, w_t \in R$
with no successive elements $w_i, w_{i+1}$
an inverse pair $(i$ mod $t)$, if $t < q$, then at least one of the products
$w_1 w_2,\dots,$ $w_{t-1} w_t$, $w_t w_1$ is freely reduced without cancellation.
\end{enumerate}
}
\end{definition}

We recall the following lemma
which concerns the word $u_r$ defined in Lemma~\ref{presentation}.

\begin{lemma} [{\cite[Lemma~5.3]{lee_sakuma}}]
\label{lem:maximal_piece}
Suppose that $r=[m_1, \dots, m_k]$ is a rational number with $0<r<1$,
and let $S(r)=(S_1, S_2, S_1, S_2)$ be as in Lemma~{\rm \ref{lem:sequence}}.
Decompose
\[
u_r \equiv v_1 v_2 v_3 v_4,
\]
where $S(v_1)=S(v_3)=S_1$ and $S(v_2)=S(v_4)=S_2$.
Then the following hold.
\begin{enumerate}[\indent \rm (1)]
\item If $k=1$, then the following hold.

\begin{enumerate}[\rm (a)]
\item No piece can contain $v_2$ or $v_4$.

\item No piece is of the form
$v_{2e} v_{4b}$ or $v_{4e} v_{2b}$,
where $v_{ib}$ and $v_{ie}$ are nonempty initial and terminal subwords of $v_i$, respectively.

\item Every subword of the form $v_{2b}$, $v_{2e}$, $v_{4b}$, or $v_{4e}$ is a piece,
where $v_{ib}$ and $v_{ie}$ are nonempty initial and terminal subwords of $v_i$ with $|v_{ib}|, |v_{ie}| \le |v_i|-1$, respectively.
\end{enumerate}

\item If $k \ge 2$, then the following hold.

\begin{enumerate}[\rm (a)]
\item No piece can contain $v_1$ or $v_3$.

\item No piece is of the form
$v_{1e} v_2 v_{3b}$ or $v_{3e} v_4 v_{1b}$,
where $v_{ib}$ and $v_{ie}$ are nonempty initial and terminal subwords of $v_i$, respectively.

\item Every subword of the form $v_{1e} v_2$, $v_2 v_{3b}$, $v_{3e} v_4$, or $v_4 v_{1b}$ is a piece,
where $v_{ib}$ and $v_{ie}$ are nonempty initial and terminal subwords of $v_i$ with $|v_{ib}|, |v_{ie}| \le |v_i|-1$, respectively.
\end{enumerate}
\end{enumerate}
\end{lemma}

In the following lemma, we mean by a {\it subsequence} a subsequence without leap.
Namely a sequence $(a_1,a_2,\dots, a_p)$ is called a {\it subsequence} of a cyclic sequence,
if there is a sequence $(b_1,b_2,\dots, b_t)$
representing the cyclic sequence
such that $p\le t$ and $a_i=b_i$ for $1\le i\le p$.

\begin{lemma}
\label{lem:maximal_piece2}
Suppose that
$r=[m_1, \dots, m_k]$
is a rational number with $0<r<1$
and that $n \ge 2$ is an integer.
Let $S(r)=(S_1, S_2, S_1, S_2)$ be as in Lemma~{\rm \ref{lem:sequence}}.
Then the following hold.
\begin{enumerate}[\indent \rm (1)]
\item
The cyclic word $(u_r^n)$ is not a product of $t$ pieces with $t\le 4n-1$.
\item
Let $w$ be a subword of the cyclic word $(u_r^n)$
which is a product of $4n-1$ pieces
but is not a product of $t$ pieces with $t<4n-1$.
Then $w$ contains a subword, $w'$,
such that $S(w')=((2n-1)\langle S_1,S_2\rangle, \ell)$
or $S(w')=(\ell, (2n-1)\langle S_2,S_1\rangle)$,
where $\ell\in\ZZ_+$.
\item
For a rational number $s$ with $0 \le s \le 1$,
suppose that the cyclic word $(u_s)$ contains a subword, $w$, as in {\rm (2)}.
Then $0<s<1$ and the following hold.
\begin{enumerate}[\rm (a)]
\item
If $k=1$, then $CS(s)$ contains $((2n-2) \langle m \rangle)$ as a subsequence.
\item
If $k \ge 2$, then $CS(s)$ contains $((2n-1) \langle S_1, S_2 \rangle)$ or
$((2n-1) \langle S_2, S_1 \rangle)$ as a subsequence.
\end{enumerate}
\end{enumerate}
\end{lemma}

\begin{proof}
The first two assertions are nothing other than \cite[Lemma~4.3]{lee_sakuma_7}.
The last assertion follows from the second assertion as follows.
Suppose that $(u_s)$ satisfies the assumption of the assertion.
Then, by the second assertion, $(u_s)$ contains a subword, $w'$,
such that $S(w')=((2n-1)\langle S_1,S_2\rangle, \ell)$
or $S(w')=(\ell, (2n-1)\langle S_2,S_1\rangle)$,
where $\ell\in\ZZ_+$.
By Remark~\ref{rem:epsilon}, we have $s\ne 0, 1$.
In the following, we assume $S(w')=((2n-1)\langle S_1,S_2\rangle, \ell)$.
(The other case can be treated similarly.)
If $k=1$, then $S_1=\emptyset$ and $S_2=(m)$ by Lemma~\ref{lem:sequence},
and therefore $S(w')=((2n-1)\langle m\rangle, \ell)$.
Since $w'$ is a subword of $(u_s)$,
the subsequence of $S(w')$ obtained by deleting the
first and the last components is a subsequence of $CS(s)$.
Hence $CS(s)$ contains $((2n-2) \langle m \rangle)$ as a subsequence.
If $k\ge 2$, then we see by Lemma~\ref{lem:sequence}
that $S(w')$ consists of $m$ and $m+1$,
and $S(w')$ begins with $m+1$.
On the other hand, $CS(s)$ consists of at most two integers
by Lemma~\ref{lem:properties}.
Hence, the first component, $m+1$, of $S(w')$ must be a component of $CS(s)$
and therefore $CS(s)$ contains $((2n-1)\langle S_1,S_2\rangle)$ as a subsequence.
\end{proof}

The following proposition enables us to
apply the small cancellation theory to our problem.

\begin{proposition}[{\cite[Proposition~4.4]{lee_sakuma_7}}]
\label{prop:small_cancellation_condition_heckoid}
Suppose that $r$ is a rational number with $0 < r< 1$
and that $n$ is an integer with $n \ge 2$.
Let $R$ be the symmetrized subset of $F(a, b)$ generated
by the single relator $u_{r}^n$ of the upper presentation of
$\Hecke(r;n)$.
Then $R$ satisfies $C(4n)$ and $T(4)$.
\end{proposition}

Now we want to investigate the geometric consequences of
Proposition~\ref{prop:small_cancellation_condition_heckoid}.
Let us begin with necessary definitions and notation following \cite{lyndon_schupp}.
A {\it map} $M$ is a finite $2$-dimensional cell complex
embedded in $\RR^2$.
To be precise, $M$ is a finite collection of vertices ($0$-cells), edges ($1$-cells),
and faces ($2$-cells) in $\RR^2$ satisfying the following conditions.
\begin{enumerate}[\indent \rm (i)]
\item A vertex is a point in $\RR^2$.

\item An edge $e$ is homeomorphic to an open interval
such that $\bar e=e\cup\{ a\}\cup \{b\}$,
where $a$ and $b$ are vertices of $M$ which are possibly identical.

\item For each face $D$ of $M$,
there is a continuous map $f$ from the
$2$-ball $B^2$ to $\RR^2$ such that
\begin{enumerate}
\item the restriction of $f$ to the interior of $B^2$
is a homeomorphism onto $D$, and

\item the image of $\partial B^2$ is equal to
$\cup_{i=1}^t \bar e_i$ for some set $\{e_1,\dots, e_t\}$ of edges of $M$.
\end{enumerate}
\end{enumerate}
The underlying space of $M$, i.e., the union of the cells in $M$,
is also denoted by the same symbol $M$.
The boundary (frontier), $\partial M$, of $M$ in $\RR^2$
is regarded as a $1$-dimensional subcomplex of $M$.
An edge may be traversed in either of two directions.
If $v$ is a vertex of a map $M$, $d_M(v)$, the {\it degree of $v$},
denotes the number of oriented edges in $M$ having $v$ as initial vertex.
A vertex $v$ of $M$ is called an {\it interior vertex}
if $v\not\in \partial M$, and an edge $e$ of $M$ is called
an {\it interior edge} if $e\not\subset \partial M$.

A {\it path} in $M$ is a sequence of oriented edges $e_1, \dots, e_t$ such that
the initial vertex of $e_{i+1}$ is the terminal vertex of $e_i$ for
every $1 \le i \le t-1$. A {\it cycle} is a closed path, namely
a path $e_1, \dots, e_t$
such that the initial vertex of $e_1$ is the terminal vertex of $e_t$.
If $D$ is a face of $M$, any cycle of minimal length which includes
all the edges of the boundary, $\partial D$, of $D$
going around once along the boundary of $D$
is called a {\it boundary cycle} of $D$.
To be precise it is defined as follows.
Let $f: B^2 \rightarrow D$ be a continuous map satisfying
condition (iii) above.
We may assume that $\partial B^2$ has a cellular structure
such that the restriction of $f$ to each cell is a homeomorphism.
Choose an arbitrary orientation of $\partial B^2$, and let
$\hat e_1, \dots, \hat e_t$ be the oriented edges of $\partial B^2$,
which are oriented in accordance with the orientation of $\partial B^2$
and which lie on $\partial B^2$ in this cyclic order with respect to the orientation of $\partial B^2$.
Let $e_i$ be the orientated edge $f(\hat e_i)$ of $M$.
Then the cycle $e_1, \dots, e_t$ is a boundary cycle of $D$.
By $d_M(D)$, the {\it degree of $D$}, we denote
the number of unoriented edges in a boundary cycle of $D$.

\begin{definition}
{\rm A non-empty map $M$ is called a {\it $[p, q]$-map} if the following conditions hold.
\begin{enumerate}[\indent \rm (i)]
\item $d_M(v) \ge p$ for every interior vertex $v$ in $M$.

\item $d_M(D) \ge q$ for every face $D$ in $M$.
\end{enumerate}
}
\end{definition}

\begin{definition}
{\rm Let $R$ be a symmetrized subset of $F(X)$. An {\it $R$-diagram} is
a pair $(M,\phi)$ of
a map $M$ and a function $\phi$ assigning to each oriented edge $e$ of $M$, as a {\it label},
a reduced word $\phi(e)$ in $X$ such that the following hold.
\begin{enumerate}[\indent \rm (i)]
\item If $e$ is an oriented edge of $M$ and $e^{-1}$ is the oppositely oriented edge,
then $\phi(e^{-1})=\phi(e)^{-1}$.

\item For any boundary cycle $\delta$ of any face of $M$,
$\phi(\delta)$ is a cyclically reduced word
representing an element of $R$.
(If $\alpha=e_1, \dots, e_t$ is a path in $M$, we define $\phi(\alpha) \equiv \phi(e_1) \cdots \phi(e_t)$.)
\end{enumerate}
We denote an $R$-diagram $(M,\phi)$ simply by $M$.
}
\end{definition}

Let $D_1$ and $D_2$ be faces (not necessarily distinct) of $M$
with an edge $e \subseteq \partial D_1 \cap \partial D_2$.
Let $e \delta_1$ and $\delta_2e^{-1}$ be boundary cycles of $D_1$ and $D_2$, respectively.
Let $\phi(\delta_1)=f_1$ and $\phi(\delta_2)=f_2$. An $R$-diagram $M$
is said to be {\it reduced}
if one never has $f_2=f_1^{-1}$.
It should be noted that if $M$ is reduced
then $\phi(e)$ is a piece for every interior edge $e$ of $M$.

As explained in \cite[Convention~1]{lee_sakuma},
we may assume the following convention.

\begin{convention}
\label{convention}
{\rm
Suppose that $r$ is a rational number with $0 < r< 1$
and that $n$ is an integer with $n \ge 2$.
Let $R$ be the symmetrized subset of $F(a, b)$
generated by the single relator $u_r^n$ of the upper presentation of $\Hecke(r;n)$.
For any reduced $R$-diagram $M$,
we assume that $M$ satisfies the following.
\begin{enumerate}[\indent \rm (1)]
\item $d_M(v) \ge 3$ for every interior vertex $v$ of $M$.

\item For every edge $e$ of $\partial M$,
the label $\phi(e)$ is a piece.

\item For a path $e_1, \dots, e_t$ in $\partial M$ of length $t\ge 2$
such that the vertex $\bar{e}_i\cap \bar{e}_{i+1}$
has degree $2$ for $i=1,2,\dots, t-1$,
$\phi(e_1) \phi(e_2) \cdots\phi(e_t)$ cannot be expressed
as a product of fewer than $t$ pieces.
\end{enumerate}
}
\end{convention}

The following corollary is immediate from Proposition~\ref{prop:small_cancellation_condition_heckoid}
and Convention~\ref{convention}.

\begin{corollary}[{\cite[Corollary~4.9]{lee_sakuma_7}}]
\label{small_cancellation_condition_2}
Suppose that $r$ is a rational number with $0<r<1$,
and that $n$ is an integer with $n \ge 2$.
Let $R$ be the symmetrized subset of $F(a, b)$
generated by the single relator $u_r^n$ of the upper presentation
of $\Hecke(r;n)$.
Then every reduced $R$-diagram is a $[4, 4n]$-map.
\end{corollary}

We turn to interpreting conjugacy in terms of diagrams.

\begin{definition}
\label{def:annular_map}
{\rm
An {\it annular map} $M$ is a connected map such that $\RR^2-M$ has exactly two connected components.
It is said to be {\it nontrivial}
if it contains a $2$-cell.
For a symmetrized subset $R$ of $F(a, b)$,
an {\it annular $R$-diagram} is an $R$-diagram
whose underlying map is an annular map.
}
\end{definition}

Let $M$ be an annular $R$-diagram, and let $K$ and $H$ be, respectively,
the unbounded and bounded components of $\RR^2-M$. We call
$\partial K (\subset \partial M)$
the {\it outer boundary} of $M$, while
$\partial H (\subset \partial M)$
is called the {\it inner boundary} of $M$.
Clearly, the {\it boundary} of $M$, $\partial M$, is the union of the outer boundary and the inner boundary.
A cycle of minimal length which contains all the edges in the outer (inner, resp.)
boundary of $M$ going around once along the boundary of $K$ ($H$, resp.)
is an {\it outer {\rm (}inner, {\rm resp.)} boundary cycle} of $M$.
An {\it outer {\rm (}inner, {\rm resp.)} boundary label of $M$} is defined to be a word $\phi(\alpha)$ in $X$
for $\alpha$ an outer (inner, resp.) boundary cycle of $M$.

\begin{convention}
\label{convention2}
{\rm
Since $M$ is embedded in $\RR^2$,
each $2$-cell of $M$ inherits an orientation of $\RR^2$.
Throughout this paper,
we assume, unlike the usual orientation convention, that
$\RR^2$ is oriented so that the boundary cycles of the $2$-cells of $M$
are clockwise.
Thus the outer boundary cycles are clockwise
and inner boundary cycles are counterclockwise,
unlike the convention in \cite[p.253]{lyndon_schupp}.
}
\end{convention}

The following lemma is a well-known classical result in combinatorial group theory.

\begin{lemma} [{\cite[Lemmas~V.5.1 and V.5.2]{lyndon_schupp}}]
\label{lem:lyndon_schupp}
Suppose $G=\langle X \,|\, R \, \rangle$ with $R$ being symmetrized.
Let $u, v$ be two cyclically reduced words in $X$
which are not trivial in $G$ and which are not conjugate in $F(X)$.
Then $u$ and $v$ represent conjugate elements in $G$ if and only if
there exists a reduced nontrivial annular $R$-diagram $M$ such that
$u$ is an outer boundary label and $v^{-1}$ is an inner boundary label of $M$.
\end{lemma}

\subsection{Structure theorem and its corollary}

We recall the following lemma obtained
from the arguments of \cite[Theorem~V.3.1]{lyndon_schupp}.

\begin{lemma}
\label{lem:inequality}
Let $M$ be an arbitrary connected map, and let $h$ denote the number of holes of $M$. Then
\[
4-4h \le \sum_{v \in \partial M} (3-d_M(v))+ \sum_{v \in M -\partial M} (4-d_M(v))+ \sum_{D \in M} (4-d_M(D)).
\]
In particular, if $M$ is a $[4,4n]$-map, then
\[
4-4h \le \sum_{v \in \partial M} (3-d_M(v)) + \sum_{D \in M} (4-4n).
\]
\end{lemma}

In the above lemma and throughout this paper,
the symbol $v \in X$ ($D \in X$, resp.)
under the symbol $\sum$, where $X$ is a map $M$ or a subspace of a map $M$,
means that the sum is over the vertices $v$ (the faces $D$, resp.)
of the map $M$ contained in the subspace $X$.

The following proposition will play an essential role in the proof of the structure theorem.

\begin{proposition}
\label{prop:key}
Let $M$ be an arbitrary connected $[4,4n]$-map.
Put
\[
\begin{aligned}
h= &\ \text{\rm the number of holes of $M$;} \\
V= &\ \text{\rm the number of vertices of $M$;} \\
E= &\ \text{\rm the number of (unoriented) edges of $M$;} \\
F= &\ \text{\rm the number of faces of $M$.}
\end{aligned}
\]
Also put
\[
\begin{aligned}
A= &\ \text{\rm the number of vertices $v$ in $\partial M$ such that $d_M(v)=2$;} \\
B= &\ \text{\rm the number of vertices $v$ in $\partial M$ such that $d_M(v) \ge 4$;} \\
C= &\ \text{\rm the number of vertices $v$ in $\partial M$ such that $d_M(v)=3$.}
\end{aligned}
\]
Then the following hold, where $\lceil x \rceil$ is the smallest integer not less than $x$.
\begin{enumerate} [\indent \rm (1)]
\item $F \ge B+ \lceil C/2 \rceil +1-h$.

\item $A \ge (4n-3)B+(4n-4) \lceil C/2 \rceil +4(1-h)n$.
\end{enumerate}

\end{proposition}

\begin{proof}
(1) Since $M$ is a $[4,4n]$-map, every interior vertex of $M$ has degree at least $4$.
So we have
\[
E \ge 1/2\{2A+3C+4(V-A-C)\}=2V-A-C/2.
\]
This inequality together with Euler's formula $V-E+F=1-h$ yields
$V+F \ge 2V-A-C/2+1-h$, so that
\[
\begin{aligned}
F &\ge V-A-C/2+1-h \\
&\ge (A+B+C)-A-C/2+1-h\\
&=B+C/2+1-h.
\end{aligned}
\]
Since $F$ is an integer, we finally have
\[
F \ge B+\lceil C/2 \rceil +1-h,
\]
as required.

(2) By Lemma~\ref{lem:inequality}, we have
\[
\begin{aligned}
4-4h &\le \sum_{v \in \partial M} (3-d_M(v)) + \sum_{D \in M} (4-4n) \\
&=\sum_{v \in \partial M} (3-d_M(v)) + F(4-4n),
\end{aligned}
\]
so that
\[
\sum_{v \in \partial M} (3-d_M(v)) \ge F(4n-4)+4-4h.
\]
Here, since $A-B \ge \sum_{v \in \partial M} (3-d_M(v))$,
we have, by (1),
\[
\begin{aligned}
A-B &\ge F(4n-4)+4-4h \\
& \ge (B+\lceil C/2 \rceil +1-h)(4n-4)+4-4h \\
&=(4n-4)B+(4n-4)\lceil C/2 \rceil + 4(1-h)n,
\end{aligned}
\]
so that $A \ge (4n-3)B+(4n-4)\lceil C/2 \rceil+4(1-h)n$,
as required.
\end{proof}

\begin{remark}
\label{rem:big_degree_vertex}
{\rm
In Proposition~\ref{prop:key}(2), if the equality holds,
then the following hold.
\begin{enumerate}
\item $V=A+B+C$, that is, there is no vertex of $M$ of degree $1$
and every vertex of $M$ lies in $\partial M$;

\item $\sum_{v \in \partial M} (3-d_M(v))=A-B$, that is,
every vertex of $\partial M$ has degree $2$, $3$ or $4$;

\item $d_M(D)=4n$ for every face $D$ of $M$.
\end{enumerate}
}
\end{remark}

Now we obtain the following strong structure theorem.

\begin{theorem}[Structure Theorem]
\label{thm:annular_structure}
Suppose that $r=[m_1, \dots, m_k]$ is a rational number with $0<r<1$,
and that $n$ is an integer with $n \ge 2$.
Let $R$ be the symmetrized subset of $F(a, b)$
generated by the single relator $u_r^n$ of the upper presentation
of $\Hecke(r;n)$,
and let $S(r)=(S_1, S_2, S_1, S_2)$ be as in Lemma~{\rm \ref{lem:sequence}}.
Suppose that $M$ is a reduced nontrivial annular $R$-diagram such that,
for $\alpha$ and $\delta$ which are, respectively, arbitrary
outer and inner boundary cycles of $M$,
\begin{enumerate}[\indent \rm (i)]
\item
$(\phi(\alpha))\equiv (u_s)$ and $(\phi(\delta))\equiv (u_{s'}^{\pm 1})$
for some rational numbers $s$ and $s'$ with $0 \le s, s' \le 1$,
\item
if $k=1$, then $CS(\phi(\alpha))$ and $CS(\phi(\delta))$
do not contain $((2n-2) \langle m \rangle)$ as a subsequence,
whereas if
$k \ge 2$, then $CS(\phi(\alpha))$ and $CS(\phi(\delta))$
do not contain $((2n-1) \langle S_1, S_2 \rangle)$ nor
$((2n-1) \langle S_2, S_1 \rangle)$ as a subsequence.
\end{enumerate}
Let the outer and inner boundaries of $M$ be denoted by
$\sigma$ and $\tau$, respectively.
Then the following hold.
\begin{enumerate} [\indent \rm (1)]
\item The outer and inner boundaries $\sigma$ and $\tau$ are simple,
i.e., they are homeomorphic to the circle,
and there is no edge contained in $\sigma \cap \tau$.

\item Every vertex of $M$ lies in $\partial M$.

\item $d_M(v)=2$ or $4$ for every vertex $v$ of $\partial M$.

\item In particular, if $\sigma \cap \tau = \emptyset$,
then between any two vertices of degree $4$
there should occur exactly $4n-3$ vertices of degree $2$
on both $\sigma$ and $\tau$,
and $d_M(D)=4n$ for every face $D$ of $M$.
\end{enumerate}
\end{theorem}

Before proving Theorem~\ref{thm:annular_structure}, we prepare the following lemma.

\begin{lemma}
\label{lem:degree2_vertices}
Under the assumption of Theorem~{\rm \ref{thm:annular_structure}},
the outer and inner boundaries $\sigma$ and $\tau$ are simple.
\end{lemma}

\begin{proof}
Suppose on the contrary that $\sigma$ or $\tau$ is not simple.
Then there is an extremal disk, say $J$, which is properly
contained in $M$ and connected to the rest of $M$ by a single vertex,
say $v_0$.
Here, recall that an {\it extremal disk} of a map $M$
is a submap of $M$ which is topologically a disk
and which has a boundary cycle $e_1, \dots, e_t$
such that the edges $e_1, \dots, e_t$ occur in order
in some boundary cycle of the whole map $M$.
By Corollary~\ref{small_cancellation_condition_2},
$M$ is a connected annular $[4,4n]$-map.
Then $J$ is a connected and simply-connected $[4,4n]$-map.

By assumption (i),
there is no vertex of degree $1$ nor $3$ in $\partial M$.
So every vertex in $\partial J-\{v_0\} \subseteq \partial M$ has degree $2$ or at least $4$.
Put
\[
\begin{aligned}
A= &\ \text{\rm the number of vertices $v$ in $\partial J-\{v_0\}$ such that $d_J(v)=2$;} \\
B= &\ \text{\rm the number of vertices $v$ in $\partial J-\{v_0\}$ such that $d_J(v) \ge 4$.}
\end{aligned}
\]
By assumptions (i) and (ii)
together with Lemma~\ref{lem:maximal_piece2}(3),
the word $\phi(\partial J |_{v_0})$
does not contain any subword of the cyclic word $(u_r^n)$
which is a product of $4n-1$ pieces
but is not a product of less than $4n-1$ pieces,
where $\partial J|_{v_0}$ denotes a boundary cycle of $J$
starting from the vertex $v_0$.
This implies by Convention~\ref{convention}(3)
that there are no $4n-2$ consecutive degree $2$ vertices in $\partial J-\{v_0\}$.
Hence we have $A \le (4n-3)(B+1)$.
We will derive a contradiction to this inequality
using Proposition~\ref{prop:key}(2) applied to $J$.

Clearly $d_J(v_0) \ge 2$.
First if $d_J(v_0)=2$, then
\[
A+1 \ge (4n-3)B+4n=(4n-3)(B+1)+3,
\]
so that $A \ge (4n-3)(B+1)+2$,
contrary to $A \le (4n-3)(B+1)$.
Next if $d_J(v_0)=3$, then
\[
\begin{aligned}
A &\ge (4n-3)B+(4n-4)\lceil 1/2 \rceil+4n=(4n-3)B+(4n-4)+4n \\
&=(4n-3)(B+1)+4n-1,
\end{aligned}
\]
contrary to $A \le (4n-3)(B+1)$.
Finally if $d_J(v_0) \ge 4$, then
\[
A \ge (4n-3)(B+1)+4n,
\]
contrary to $A \le (4n-3)(B+1)$.
\end{proof}

\begin{proof}[Proof of Theorem~{\rm \ref{thm:annular_structure}}]
Arguing as in the proof of Lemma~\ref{lem:degree2_vertices},
$M$ is a connected annular $[4,4n]$-map
such that every vertex in $\partial M$ has degree $2$ or at least $4$.
We divide into two cases.

\medskip
\noindent {\bf Case 1.} $\sigma \cap \tau= \emptyset$.
\medskip

In this case, (1) follows immediately
from Lemma~\ref{lem:degree2_vertices}.

(2)--(4)
Put
\[
\begin{aligned}
A_1= &\ \text{\rm the number of vertices $v$ in $\sigma$ such that $d_M(v)=2$;} \\
A_2= &\ \text{\rm the number of vertices $v$ in $\tau$ such that $d_M(v)=2$;}\\
B_1= &\ \text{\rm the number of vertices $v$ in $\sigma$ such that $d_M(v) \ge 4$;} \\
B_2= &\ \text{\rm the number of vertices $v$ in $\tau$ such that $d_M(v) \ge 4$.} \\
\end{aligned}
\]
Since $\partial M$ is the disjoint union of $\sigma$ and $\tau$,
Proposition~\ref{prop:key}(2) applied to $M$ yields
\[
A_1+A_2 \ge (4n-3)(B_1+B_2)=(4n-3)B_1+(4n-3)B_2.
\]
Here, we observe that there are no $4n-2$ consecutive degree $2$ vertices on
$\sigma$ nor on $\tau$.
In fact, if this is not the case,
then by Convention~\ref{convention}(3),
$(\phi(\alpha))$ or $(\phi(\delta))$
contains a subword, $w$, which is a product of $4n-1$ pieces
but is not a product of $t$ pieces with $t<4n-1$.
But, this contradicts
assumptions (i) and (ii)
by Lemma~\ref{lem:maximal_piece2}(3).
Hence we have $A_1 \le (4n-3)B_1$ and $A_2 \le (4n-3)B_2$.
Thus the above inequality is actually an equality.
By Remark~\ref{rem:big_degree_vertex},
every vertex of $M$ lies in $\partial M$,
every vertex in $\partial M$ has degree $2$ or $4$,
and $d_M(D)=4n$ for every face $D$ of $M$.
Moreover, since $A_1=(4n-3)B_1$ and $A_2=(4n-3)B_2$
and since there are no $4n-2$ consecutive degree $2$ vertices on both $\sigma$ and $\tau$,
between any two vertices of degree $4$
there should occur exactly $4n-3$ vertices of degree $2$
on both $\sigma$ and $\tau$.
Therefore (2)--(4) hold.

\medskip
\noindent {\bf Case 2.} $\sigma \cap \tau \neq \emptyset$.
\medskip

(1) Suppose on the contrary that $\sigma \cap \tau$ contains an edge.
As illustrated in Figure~\ref{fig.island},
there is a submap $J$ of $M$ such that
\begin{enumerate}[\indent \rm (i)]
\item $J$ is bounded by a simple
closed path of the form $\sigma_1 \tau_1$, where $\sigma_1 \subseteq \sigma$
and $\tau_1 \subseteq \tau$;

\item $J$ is connected to the rest of $M$ by two distinct vertices, say $v_1$ and $v_2$,
where $\sigma_1 \cap \tau_1=\{v_1, v_2 \}$
and $v_1$ is an endpoint of an edge contained in $\sigma \cap \tau$.
Note that $d_J(v_1)=d_M(v_1)-1 \ge 3$ and $d_J(v_2)\ge 2$.
\end{enumerate}
Then $J$ is a connected and simply connected $[4,4n]$-map
such that every vertex in $\partial J-\{v_1, v_2\}$ has degree $2$ or at least $4$.
Put
\[
\begin{aligned}
A= &\ \text{\rm the number of vertices $v$ in $\partial J-\{v_1,v_2\}$ such that $d_J(v)=2$;}\\
B= &\ \text{\rm the number of vertices $v$ in $\partial J-\{v_1,v_2\}$ such that $d_J(v) \ge 4$.}
\end{aligned}
\]
Arguing as in Case~1,
there are no $4n-2$ consecutive degree $2$ vertices
on $\sigma_1-\{v_1, v_2\}$ nor on $\tau_1-\{v_1, v_2\}$.
Hence we have $A \le (4n-3)(B+2)$.
We will derive a contradiction to this inequality
using Proposition~\ref{prop:key}(2) applied to $J$.

First if $d_J(v_1)=3$ and $d_J(v_2)=2$, then
\[
\begin{aligned}
A+1 & \ge (4n-3)B+(4n-4)\lceil 1/2 \rceil+4n=(4n-3)B+(4n-4)+4n \\
& =(4n-3)(B+2)+2,
\end{aligned}
\]
so that $A \ge (4n-3)(B+2)+1$,
contrary to $A \le (4n-3)(B+2)$.
Second if $d_J(v_1) \ge 4$ and $d_J(v_2)=2$, then
\[
A+1 \ge (4n-3)(B+1)+4n=(4n-3)(B+2)+3,
\]
so that $A \ge (4n-3)(B+2)+2$,
contrary to $A \le (4n-3)(B+2)$.
Third if $d_J(v_1)=3$ and $d_J(v_2)=3$, then
\[
\begin{aligned}
A & \ge (4n-3)B+(4n-4)\lceil 1 \rceil+4n=(4n-3)B+(4n-4)+4n \\
& =(4n-3)(B+2)+2,
\end{aligned}
\]
contrary to $A \le (4n-3)(B+2)$.
Fourth either if $d_J(v_1)\ge 4$ and $d_J(v_2)=3$
or if $d_J(v_1)=3$ and $d_J(v_2)\ge 4$, then
\[
\begin{aligned}
A & \ge (4n-3)(B+1)+(4n-4)\lceil 1/2 \rceil+4n=(4n-3)(B+1)+(4n-4)+4n \\
& =(4n-3)(B+2)+4n-1,
\end{aligned}
\]
contrary to $A \le (4n-3)(B+2)$.
Finally if $d_J(v_1)\ge 4$ and $d_J(v_2) \ge 4$, then
\[
A \ge (4n-3)(B+2)+4n,
\]
contrary to $A \le (4n-3)(B+2)$.

\begin{figure}[h]
\begin{center}
\includegraphics{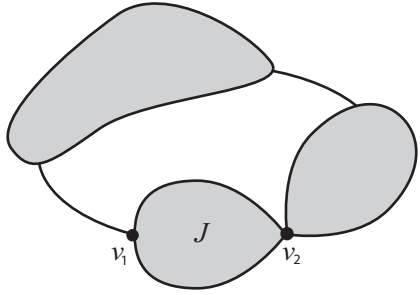}
\end{center}
\caption{\label{fig.island}
A possible annular map $M$ when $\sigma \cap \tau$ contains an edge.}
\end{figure}

(2)--(3) By (1), $\sigma \cap \tau$ consists of finitely many vertices in $M$.
First suppose that $\sigma \cap \tau$ consists of at least two vertices,
say $v_1, \dots, v_t$, where $t \ge 2$ and where these vertices are indexed
so that there is a submap $J_i$ of $M$ for every $i=1, \dots, t$ such that
\begin{enumerate}[\indent \rm (i)]
\item $J_i$ is bounded by a simple
closed path of the form $\sigma_i \tau_i$, where $\sigma_i \subseteq \sigma$
and $\tau_i \subseteq \tau$;

\item $J_i$ is connected to the rest of
$M$ by two distinct vertices, say $v_i$ and $v_{i+1}$,
where $\sigma_i \cap \tau_i=\{v_i, v_{i+1} \}$ and where
$d_{J_i}(v_i), d_{J_i}(v_{i+1}) \ge 2$ and
$d_{J_i}(v_{i+1})+d_{J_{i+1}}(v_{i+1})=d_M(v_{i+1})$
(taking the indices modulo $n$).
\end{enumerate}
Then each $J_i$ is a connected and simply connected $[4, 4n]$-map
such that $M=J_1 \cup \cdots \cup J_{t}$.
Moreover $\sigma=\sigma_1 \cup \cdots \cup \sigma_{t}$
and $\tau=\tau_1 \cup \cdots \cup \tau_{t}$.
The same argument as for $(M', v_0', v_0'')$ below applies to each
$(J_i, v_i, v_{i+1})$ to prove the assertions.

Next suppose that $\sigma \cap \tau$ consists of a single vertex,
say $v_0$. Cut $M$ open at $v_0$ to get a connected and simply connected
$[4,4n]$-map $M'$. In this process, the vertex $v_0$ is separated into two distinct vertices,
say $v_0'$ and $v_0''$, in $M'$ such that $d_{M'}(v_0'), d_{M'}(v_0'') \ge 2$
and $d_{M'}(v_0')+d_{M'}(v_0'')=d_M(v_0)$. Then $M'$ is bounded by a simple
closed path of the form $\sigma_0 \tau_0$, where $\sigma_0 \cap \tau_0=\{v_0', v_0''\}$.
Put
\[
\begin{aligned}
F= &\ \text{\rm the number of faces of $M'$;} \\
A= &\ \text{\rm the number of vertices $v$ in $\partial M'-\{v_0', v_0''\}$ such that $d_{M'}(v)=2$;} \\
B= &\ \text{\rm the number of vertices $v$ in $\partial M'-\{v_0', v_0''\}$ such that $d_{M'}(v) \ge 4$.}
\end{aligned}
\]

\medskip
\noindent {\bf Claim 1.} $d_{M'}(v_0')=d_{M'}(v_0'')=2$.

\begin{proof}[Proof of Claim~{\rm 1}]
Since every vertex in $\partial M'-\{v_0', v_0''\}$ has degree $2$ or at least $4$
and since $4n-2$ vertices of degree $2$ do not occur consecutively
on $\sigma_0-\{v_0', v_0''\}$ nor on $\tau_0-\{v_0', v_0''\}$,
we have $A \le (4n-3)(B+2)$.
Suppose on the contrary that $d_{M'}(v_0') \ge 3$ or $d_{M'}(v_0'') \ge 3$.
Without loss of generality, assume that $d_{M'}(v_0') \ge 3$.
Then repeating the same arguments as in the proof of (1)
replacing $J$, $v_1$ and $v_2$
with $M'$, $v_0'$ and $v_0''$, respectively, we obtain a
contradiction to $A \le (4n-3)(B+2)$.
\end{proof}

By Claim~1, $d_{M'}(v_0')=d_{M'}(v_0'')=2$.
Hence there exist unique $2$-cells $D_1$ and $D_2$ in $M'$
such that $v_0' \in \partial D_1$ and $v_0'' \in \partial D_2$.

\medskip
\noindent {\bf Claim 2.} $D_1=D_2$.

\begin{proof}[Proof of Claim~{\rm 2}]
Suppose on the contrary that $D_1 \neq D_2$.
Then note that the number of vertices of degree $2$ on
$(\partial D_1 \cup \partial D_2) \cap (\sigma_0 \cup \tau_0) -\{v_0', v_0''\}$
is less than or equal to $d_{M'}(D_1)+ d_{M'}(D_2)-6$.
Since every vertex in $\partial M'-\{v_0', v_0''\}$ has degree $2$ or at least $4$
and since $4n-2$ vertices of degree $2$ do not occur consecutively
on $\sigma_0-\{v_0', v_0''\}$ nor on $\tau_0-\{v_0', v_0''\}$,
we have
\[
A \le (4n-3)(B-2)+d_{M'}(D_1)+ d_{M'}(D_2)-6.
\]

On the other hand, by Lemma~\ref{lem:inequality},
\[
\begin{aligned}
4 &\le \sum_{v \in \partial M'} (3-d_{M'}(v)) + \sum_{D \in M'} (4-d_{M'}(D)) \\
& \le \sum_{v \in \partial M'} (3-d_{M'}(v)) +\sum_{D \in M'-\{D_1, D_2\}} (4-4n)+(4-d_{M'}(D_1))+(4-d_{M'}(D_2)),
\end{aligned}
\]
so that
\[
4+\sum_{D \in M'-\{D_1, D_2\}} (4n-4)
\le \sum_{v \in \partial M'} (3-d_{M'}(v))+(4-d_{M'}(D_1))+(4-d_{M'}(D_2)).
\]
By Proposition~\ref{prop:key}(1),
\[
(B-1)(4n-4) \le (F-2)(4n-4) = \sum_{D \in M'-\{D_1, D_2\}} (4n-4);
\]
hence
\[
4+(B-1)(4n-4) \le \sum_{v \in \partial M'} (3-d_{M'}(v))+(4-d_{M'}(D_1))+(4-d_{M'}(D_2)).
\]
Note that $d_{M'}(v_0')=d_{M'}(v_0'')=2$ and so
$\sum_{v \in \partial M'} (3-d_{M'}(v)) \le A+2-B$.
Hence the above inequality implies
\[
4+(B-1)(4n-4) \le (A+2-B)+(4-d_{M'}(D_1))+(4-d_{M'}(D_2)).
\]
Thus
\[
\begin{aligned}
A & \ge (B-1)(4n-4)+B-6+d_{M'}(D_1)+d_{M'}(D_2) \\
&=(4n-3)(B-2)+d_{M'}(D_1)+ d_{M'}(D_2)+4n-8,
\end{aligned}
\]
contrary to $A \le (4n-3)(B-2)+d_{M'}(D_1)+ d_{M'}(D_2)-6$.
\end{proof}

\noindent {\bf Claim 3.} $B=0$.

\begin{proof}[Proof of Claim~{\rm 3}]
Suppose on the contrary that $B \ge 1$.
Then note that the number of vertices of degree $2$ in
$\partial D_1 \cap (\sigma_0 \cup \tau_0)-\{v_0', v_0''\}$
is less than or equal to $d_{M'}(D_1)-4$.
Since every vertex in $\partial M'-\{v_0', v_0''\}$ has degree $2$ or at least $4$
and since $4n-2$ vertices of degree $2$ do not occur consecutively
on $\sigma_0-\{v_0', v_0''\}$ nor on $\tau_0-\{v_0', v_0''\}$,
we have
\[
A \le (4n-3)(B-1)+d_{M'}(D_1)-4.
\]
On the other hand, by Lemma~\ref{lem:inequality}, we have
\[
\begin{aligned}
4 &\le \sum_{v \in \partial M'} (3-d_{M'}(v)) + \sum_{D \in M'} (4-d_{M'}(D)) \\
& \le \sum_{v \in \partial M'} (3-d_{M'}(v)) +\sum_{D \in M'-\{D_1\}} (4-4n)+(4-d_{M'}(D_1)),
\end{aligned}
\]
so that
\[
4+\sum_{D \in M'-\{D_1\}} (4n-4)
\le \sum_{v \in \partial M'} (3-d_{M'}(v))+(4-d_{M'}(D_1)).
\]
By Proposition~\ref{prop:key}(1),
\[
B(4n-4) \le (F-1)(4n-4) = \sum_{D \in M'-\{D_1\}} (4n-4);
\]
hence
\[
4+B(4n-4) \le \sum_{v \in \partial M'} (3-d_{M'}(v))+(4-d_{M'}(D_1)).
\]
Here, since $d_{M'}(v_0')=d_{M'}(v_0'')=2$,
we have $\sum_{v \in \partial M'} (3-d_{M'}(v)) \le A+2-B$, so that
\[
4+B(4n-4) \le (A+2-B)+(4-d_{M'}(D_1)).
\]
Thus
\[
A \ge B(4n-4)+B-2+d_{M'}(D_1)=(4n-3)(B-1)+d_{M'}(D_1)+4n-5,
\]
contrary to
$A \le (4n-3)(B-1)+d_{M'}(D_1)-4$.
\end{proof}

By Claim~3, $M$ consists of only one $2$-cell,
thus proving (2) and (3).
\end{proof}

We define the {\it outer boundary layer} of
an annular map $M$ to be the submap of $M$
consisting of all faces $D$
such that the intersection of $\partial D$ with the outer boundary, $\sigma$,
of $M$ contains an edge,
together with the edges and vertices contained in $\partial D$.
The {\it inner boundary layer} of $M$ is defined similarly by
using the inner boundary, $\tau$, of $M$.

\begin{corollary}
\label{cor:structure}
Let $M$ be a reduced nontrivial annular
diagram over $\Hecke(r;n)=\langle a, b \svert u_r^n \rangle$
satisfying the assumptions of Theorem~{\rm \ref{thm:annular_structure}}.
Then Figure~{\rm \ref{fig.layer}} illustrates the only two possible shapes of $M$.
In particular, the following hold.
\begin{enumerate} [\indent \rm (1)]
\item
If $\sigma\cap\tau\ne\emptyset$,
then $M$ consists of a single layer,
namely, the outer and inner boundary layers coincide.
Moreover,
the number of faces of $M$ is equal to the number of degree $4$ vertices of $M$.
Here the number of faces is variable.
\item
If $\sigma\cap\tau=\emptyset$, then $M$ consists of two layers,
namely, the intersection of the outer and inner boundary layers of $M$ is a circle,
and $M$ is the union of these two layers along the circle.
Moreover,
every vertex of $M$ lies in $\partial M$, and
the number of faces of the outer {\rm (}inner, respectively{\rm )} boundary layer is equal to the number of degree
$4$ vertices contained in $\sigma$ {\rm (}$\tau$, respectively{\rm )}.
Here the number of faces per layer is variable.
\end{enumerate}
\end{corollary}

\begin{figure}[h]
\begin{center}
\includegraphics{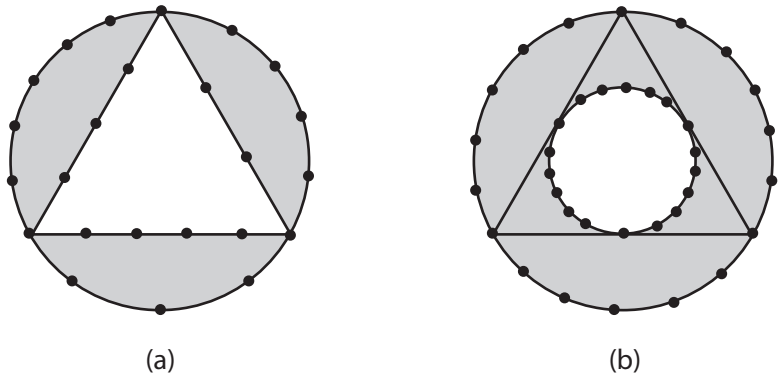}
\end{center}
\caption{\label{fig.layer}
Two possible shapes of $M$ when $n=2$}
\end{figure}

The following notation will be used in
Section~\ref{sec:proof of main theorem(1) for the case when $r=1/m$}
and \cite[Section~3]{lee_sakuma_10}.

\begin{notation}
\label{notation:cell_boundary}
{\rm
Suppose that $M$ is a connected annular map as in Figure~\ref{fig.layer},
and let $J$ be the outer boundary layer of $M$.
Choose a vertex, say $v_0$, lying in both the outer and inner boundaries of $J$,
and let $\alpha$ and $\beta$ be, respectively,
the outer and inner boundary cycles of $J$ starting from $v_0$,
where $\alpha$ is read clockwise and $\beta$ is read counterclockwise.
Let $D_1,\dots, D_t$ be the $2$-cells of $J$,
such that $\alpha$ goes through their boundaries in this order.
By the symbol $\partial D_i^{\pm}$, we denote an oriented edge path
contained in $\partial D_i$, such that
\[
\begin{aligned}
\alpha&=\partial D_1^+\cdots \partial D_t^+, \\
\beta^{-1}&=\partial D_1^-\cdots \partial D_t^-.
\end{aligned}
\]
}
\end{notation}

\section{Technical Lemmas}
\label{sec:technical_lemmas}

In the remainder of this paper,
we study the even Heckoid group $\Hecke(1/p;n)$,
where $p$ and $n$ are integers greater than $1$.
Recall that the region, $R$, bounded by a pair of
Farey edges with an endpoint $\infty$
and a pair of Farey edges with an endpoint $1/p$
forms a fundamental domain for the action of $\RGPC{1/p}{n}$ on $\HH^2$
(see Figure~\ref{fig.fd_orbifold}).
Let $I_1(1/p;n)$ and $I_2(1/p;n)$ be the (closed or half-closed) intervals in $\RR$
defined as follows:
\[
\begin{aligned}
I_1(1/p;n) &=[0, r_1), \ \mbox{where} \ r_1=[p, 2n-2], \\
I_2(1/p;n) &=[r_2, 1], \ \mbox{where} \ r_2=[p-1, 2].
\end{aligned}
\]
Then we may choose a fundamental domain $R$ so that
the intersection of $\bar R$ with $\partial \HH^2$ is equal to
the union $\bar I_1(1/p;n) \cup \bar I_2(1/p;n)\cup \{\infty,1/p\}$.

\begin{lemma}
\label{lem:connection}
For any rational number $s \in I_1(1/p;n) \cup I_2(1/p;n)$,
$CS(s)$ does not contain $((2n-2) \langle p \rangle)$ as a subsequence.
\end{lemma}

\begin{proof}
This is nothing other than \cite[Proposition~5.1(1)]{lee_sakuma_7}.
\end{proof}

As an easy consequence of
Lemmas~\ref{lem:maximal_piece2}(3)
and \ref{lem:connection},
we obtain the following.

\begin{corollary}
\label{cor:consecutive_vertices}
For any rational number $s \in I_1(1/p;n)\cup I_2(1/p;n)$,
the cyclic word $(u_s)$ cannot contain a subword
$w$ of the cyclic word $(u_{1/p}^{\pm n})$
which is a product of $4n-1$ pieces
but is not a product of less than $4n-1$ pieces.
\end{corollary}

If $\RGP{1/p}$ is the group of automorphisms of
the Farey tessellation $\DD$ generated by reflections in the edges of $\DD$ with an endpoint $1/p$,
and $\RGPP{1/p}$
is the group generated by $\RGP{1/p}$ and $\RGP{\infty}$,
then the region, $Q$, bounded by a pair of Farey edges with an endpoint $\infty$
and a pair of Farey edges with an endpoint
$1/p$ forms a fundamental domain of the action of
$\RGPP{1/p}$ on $\HH^2$.
Let $I_1(1/p)$ and $I_2(1/p)$ be the closed intervals in $\RRR$
obtained as the intersection with $\RRR$ of the closure of $Q$.
Then the intervals $I_1(1/p)$ and $I_2(1/p)$ are given by
$I_1(1/p)=\{0\}$ and $I_2(1/p)=[\frac{1}{p-1},1]$.
Clearly $I_1(1/p) \subsetneq I_1(1/p;n)$ and $I_2(1/p) \subsetneq I_2(1/p;n)$.
It was shown in {\cite[Proposition~4.6]{Ohtsuki-Riley-Sakuma}} that
if two elements $s$ and $s'$ of $\QQQ$ belong to the same $\RGPP{1/p}$-orbit,
then the unoriented loops $\alpha_s$ and $\alpha_{s'}$ are homotopic in $S^3-K(1/p)$.

\begin{lemma}
\label{lem:inside_orbit}
For any rational number $s \in I_1(1/p) \cup I_2(1/p)$,
if $s\ne 0$, then every term of $CS(s)$ is less than $p$.
\end{lemma}

\begin{proof}
The assertion is nothing other than \cite[Proposition~3.19]{lee_sakuma_2}.
\end{proof}

\begin{lemma}
\label{lem:outside_orbit}
For any rational number $s \in I_1(1/p;n) \setminus I_1(1/p)$,
$CS(s)$ contains a subsequence
$(p+c, d \langle p \rangle, p+c')$ for some $c, c'\ge 1$ and $0 \le d \le 2n-4$.
\end{lemma}

\begin{proof}
Any rational number $s \in I_1(1/p;n) \setminus I_1(1/p)$,
i.e., $0 < s < [p, 2n-2]$,
has a continued fraction expansion $s=[l_1, \dots, l_t]$,
where $t \ge 1$, $(l_1, \dots, l_t) \in (\ZZ_+)^t$ and $l_t \ge 2$,
such that
\begin{enumerate}[\indent \rm (i)]
\item $t \ge 3$, $l_1=p$ and $l_2=1$;
or

\item $t \ge 2$, $l_1=p$ and $2 \le l_2 \le 2n-3$; or

\item $t \ge 1$ and $l_1 \ge p+1$.
\end{enumerate}
If (i) happens, $CS(s)$ contains a subsequence
$(p+1, p+1)$, so the assertion holds with $d=0$.
If (ii) happens, then $\tilde s=[l_2-1,l_3,\dots,l_t]$,
where $\tilde{s}$ denotes the rational number defined as in Lemma~\ref{lem:induction1}
for the rational number $s$ so that $CS(\tilde{s})=CT(s)$
and therefore $CT(s)=CS(\tilde s)$ contains $l_2-1$ by Lemma~\ref{lem:properties}.
Hence $CS(s)$ contains a subsequence
$(p+1, d \langle p \rangle, p+1)$ with $d=l_2-1$.
Since $1 \le d=l_2-1 \le 2n-4$, the assertion holds.
If (iii) happens, $CS(s)$ contains a subsequence
$(p+c, p+c')$ for some $c, c' \ge 1$, so the assertion holds
with $d=0$.
\end{proof}

\begin{lemma}
\label{lem:outside_orbit2}
For any rational number $s \in I_2(1/p;n) \setminus I_2(1/p)$,
$CS(s)$ contains $(p-1, p, p-1)$ as a subsequence
and does not contain $(p,p)$ as a subsequence.
\end{lemma}

\begin{proof}
Any rational number $s \in I_2(1/p;n) \setminus I_2(1/p)$,
i.e., $[p-1,2] \le s < [p-1]$,
has a continued fraction expansion $s=[l_1, \dots, l_t]$,
where $t \ge 2$, $(l_1, \dots, l_t) \in (\ZZ_+)^t$,
$l_1=p-1$, $l_2 \ge 2$ and $l_t \ge 2$.
Then by Lemma~\ref{lem:properties}(2b),
$CS(s)$ consists of $p-1$ and $p$ without two consecutive terms $(p, p)$.
It then follows that $CS(s)$ contains $(p-1, p, p-1)$ as a subsequence.
\end{proof}

\section{Proof of Main Theorem~\ref{thm:conjugacy}(1)
for the case when $r=1/p$}
\label{sec:proof of main theorem(1) for the case when $r=1/m$}

Suppose on the contrary that there exist two distinct rational numbers $s$ and $s'$
in $I_1(1/p;n)\cup I_2(1/p;n)$ for which the simple loops $\alpha_s$ and $\alpha_{s'}$
are homotopic in $\orbs(1/p;n)$. Then $u_s$ and $u_{s'}^{\pm 1}$ are conjugate in $\Hecke(1/p;n)$.
By Lemma~\ref{lem:lyndon_schupp}, there is a reduced nontrivial annular
diagram $M$ over $\Hecke(1/p;n)=\langle a, b \svert u_{1/p}^n \rangle$ with
$(\phi(\alpha)) \equiv (u_s)$ and $(\phi(\delta)) \equiv (u_{s'}^{\pm 1})$,
where $\alpha$ and $\delta$ are, respectively, outer and inner boundary cycles of $M$.
Since $s, s' \in I_1(1/p;n)\cup I_2(1/p;n)$,
we see by Lemma~\ref{lem:connection}
that $CS(\phi(\alpha))$ and $CS(\phi(\delta))$
do not contain $((2n-2) \langle p \rangle)$ as a subsequence.
So by Corollary~\ref{cor:structure}, $M$ is shaped as
in Figure~\ref{fig.layer}(a) or Figure~\ref{fig.layer}(b).

\begin{lemma}
\label{lem:claim1}
$M$ is shaped as in Figure~{\rm \ref{fig.layer}(a)}, that is,
$M$ satisfies the conclusion of Corollary~{\rm \ref{cor:structure}(1)}.
\end{lemma}

\begin{proof}
Suppose on the contrary that $M$ is shaped as in Figure~\ref{fig.layer}(b).
Then $(\phi(\alpha)) \equiv (u_s)$ contains a subword of the cyclic word
$(u_{1/p}^{\pm n})$ which is a product of $4n-2$ pieces
but is not a product of less than $4n-2$ pieces
(see Convention~\ref{convention}(3) and Theorem~\ref{thm:annular_structure}(4)).
Since $4n-2 \ge 6$, this
together with Lemma~\ref{lem:maximal_piece}(1c)
implies that $CS(\phi(\alpha))=CS(s)$ contains a term $p$
and consists of more than two terms.
Thus we have $s \neq 0$, because $CS(u_0)=\lp 2 \rp$
by Remark~\ref{rem:epsilon}.
Then by Lemma~\ref{lem:inside_orbit}, $s \notin I_1(1/p) \cup I_2(1/p)$.
By Lemmas~\ref{lem:outside_orbit} and \ref{lem:outside_orbit2},
the cyclic word $(u_s)$ contains a subword $w$ for which $S(w)$ is a subsequence of $CS(s)$
such that
\[
S(w)=
\begin{cases}
(p+c, d \langle p\rangle, p+c') & \text{if $s \in I_1(1/p;n)\backslash I_1(1/p)$};\\
(p-1, p, p-1) & \text{if $s \in I_2(1/p;n)\backslash I_2(1/p)$},
\end{cases}
\]
where $c, c' \ge 1$ and $0 \le d \le 2n-4$.

\medskip
\noindent {\bf Claim.} {\it There is a face $D$ in the outer boundary layer of $M$ such that
$\phi(\partial D^+)$ is a subword of $w$ {\rm (}recall Notation~{\rm \ref{notation:cell_boundary})}.}

\begin{proof}[Proof of Claim]
Suppose that there is no such face.
Then either (i) there is a face, $D$, in the outer boundary layer of $M$ such that
$\phi(\partial D^+)\equiv uwv$ for some words $u$ and $v$
such that at least one of them is nonempty, or
(ii) there are two successive faces, say $D_1$ and $D_2$,
in the outer boundary layer of $M$
such that $\phi(\partial D^+_1)\equiv uw_1$ and
$\phi(\partial D^+_2)\equiv w_2v$,
where $u$, $v$, $w_1$ and $w_2$ are nonempty words such that
$w\equiv w_1w_2$.
If (i) holds, then by using the fact that $S(w)$ is a subsequence of $CS(s)$,
we see that
the first or the last component of $S(w)$ is also a component of
$CS(\phi(\partial D))=\lp 2n \langle p \rangle \rp$, a contradiction.
If (ii) holds, then again by using the fact that $S(w)$ is a subsequence of $CS(s)$,
we see that either
the first component of $S(w)$ is also a component of
$CS(\phi(\partial D_1))=\lp 2n \langle p \rangle \rp$
or the last component of $S(w)$ is also a component of
$CS(\phi(\partial D_2))=\lp 2n \langle p \rangle \rp$, a contradiction.
\end{proof}

For such a face $D$ as in the statement of the above claim,
since $CS(\phi(\partial D))=\lp 2n \langle p \rangle \rp$,
this claim
yields that $S(\phi(\partial D^-))$ must contain
$(p, p)$ or $(\ell, (2n-3) \langle p \rangle, \ell')$ as a subsequence
for some $\ell, \ell' \in \ZZ_+$.
But then by Lemma~\ref{lem:maximal_piece}(1),
the word $\phi(\partial D^-)$ cannot be expressed as a product of
$2$ pieces of $(u_r^{\pm 1})$,
contradicting Figure~\ref{fig.layer}(b)
(cf. Corollary~\ref{cor:structure}(2)).
\end{proof}

\begin{lemma}
\label{lem:after_claim1}
For every face $D$ in $M$, $S(\phi(\partial D^{\pm}))$ contains a term $p$.
\end{lemma}

\begin{proof}
Suppose that $S(\phi(\partial D^+))$ does not contain a term $p$.
Then $S(\phi(\partial D^+))$ is of the form either
$(\ell)$ with $1 \le \ell \le p-1$ or
$(\ell_1, \ell_2)$ with $1 \le \ell_1, \ell_2 \le p-1$.
In the first case, $\phi(\partial D^-)$ is a product of $4n-1$ pieces,
but is not a product of less than $4n-1$ pieces.
But since $\phi(\partial D^-)$ is a subword of
$(\phi(\delta)) \equiv (u_{s'}^{\pm 1})$, this gives a contradiction
to Corollary~\ref{cor:consecutive_vertices}.
In the second case,
$S(\phi(\partial D^-))=(p-\ell_1, (2n-2)\langle p \rangle, p-\ell_2)$,
which implies that $CS(\phi(\delta))=CS(s')$ contains $((2n-2)\langle p \rangle)$
as a subsequence, contrary to Lemma~\ref{lem:connection}.

The same argument applies to $S(\phi(\partial D^-))$.
\end{proof}

\begin{lemma}
\label{lem:after_claim2}
$s, s' \notin I_1(1/p) \cup I_2(1/p)$.
\end{lemma}

\begin{proof}
By Lemma~\ref{lem:after_claim1}, both $CS(s)$ and $CS(s')$
contain a term bigger than or equal to $p$.
If $s, s' \neq 0$, then the assertion follows from Lemma~\ref{lem:inside_orbit}.
In the remainder, we show that $s, s' \neq 0$.
Suppose that this does not hold, say $s=0$.
Then, since $CS(u_0)=\lp 2 \rp$ by Remark~\ref{rem:epsilon}
and since both $CS(s)$ and $CS(s')$ contain a term
bigger than or equal to $p$,
we see that the annular diagram consists of only one $2$-cell, $D$,
and $CS(\phi(\alpha))=CS(\phi(\partial D^+))=\lp p \rp$, where $p=2$.
Then $CS(\phi(\delta))=CS(\phi(\partial D^-))=\lp (2n-1) \langle p\rangle \rp$, a contradiction.
\end{proof}

\begin{lemma}
\label{lem:claim2}
$s, s' \notin I_2(1/p;n)$.
\end{lemma}

\begin{proof}
Suppose on the contrary that $s$ or $s'$ is contained in $I_2(1/p;n)$.
Without loss of generality, assume that $s \in I_2(1/p;n)$.
Then by Lemma~\ref{lem:after_claim2}, $s \in I_2(1/p;n) \setminus I_2(1/p)$.
By Lemma~\ref{lem:outside_orbit2},
$CS(\phi(\alpha))=CS(s)$ contains $(p-1, p, p-1)$ as a subsequence
and does not contain $(p,p)$ as a subsequence.
Let $w$ be the subword of $(u_s)$ such that
$S(w)=(p-1, p, p-1)$.
Then arguing as in the proof of the claim in the proof of Lemma~\ref{lem:claim1},
we see that there is a face $D$ such that $\phi(\partial D^+)$ is a subword of $w$,
so that $S(\phi(\partial D^-))$ contains $(\ell, (2n-3) \langle p\rangle, \ell')$
as a subsequence for some $\ell, \ell' \in \ZZ_+$.
This implies that $CS(\phi(\delta))=CS(s')$ contains $(p+1, d \langle p \rangle, p+1)$
as a subsequence for some $d \ge 2n-3$.
Let $s'=[a_1,a_2,\dots, a_t]$ be a continued fraction expansion.
Then since $CS(s')$ consists of $p$ and $p+1$, we see $a_1=p$
by Lemma~\ref{lem:properties}.
If $a_2 \ge 2$, then since $CT(s')$ contains $d$ as a component,
we see by using Lemmas~\ref{lem:properties} and \ref{lem:induction1}
that $a_2-1=d$ or both $a_2-1=d-1$
and $t \ge 3$, i.e.,
$a_2=d+1$ or both $a_2=d$ and $t \ge 3$.
Since $s' \in I_1(1/p;n)\cup I_2(1/p;n)$, we must have $a_2\le 2n-3$,
and therefore
$a_2=d=2n-3$ and $t \ge 3$.
Also if $a_2=1$, then by Lemma~\ref{lem:properties},
$CS(s')$ does not contain $(p,p)$ as a subsequence.
So we have $d=1$, and hence $a_2=d=2n-3$.
In this case, clearly $t \ge 3$.
Thus in either case, $s'=[p,2n-3, a_3, \dots, a_t]$ with $t \ge 3$,
i.e., $[p, 2n-3] < s'< [p, 2n-2]$.
Thus, by Lemmas~\ref{lem:properties} and \ref{lem:induction1},
$CS(s')$ contains $(p+1, (2n-4) \langle p \rangle, p+1)$ too
as a subsequence.
Arguing similarly
as in the proof of the claim in the proof of Lemma~\ref{lem:claim1},
we see that there is a face $D$ such that $\phi(\partial D^-)$ is a subword of
the word corresponding to the subsequence $(p+1, (2n-4) \langle p \rangle, p+1)$.
Since $CS(\phi(\partial D))=\lp 2n \langle p \rangle \rp$,
this yields that $S(\phi(\partial D^+))$ must contain
$(p, p)$ as a subsequence,
a contradiction.
\end{proof}

\begin{lemma}
\label{lem:every_term}
Every term of $CS(s)$ and $CS(s')$ is greater than or equal to $p$.
\end{lemma}

\begin{proof}
By Lemmas~\ref{lem:after_claim2} and \ref{lem:claim2},
$s, s' \in I_1(1/p;n) \setminus I_1(1/p)$.
Then by Lemma~\ref{lem:outside_orbit},
$CS(s)$ and $CS(s')$ contain a term greater than $p$.
So by Lemma~\ref{lem:properties},
every term of $CS(s)$ and $CS(s')$ is greater than or equal to $p$.
\end{proof}

\begin{lemma}
\label{lem:ending_cliam1}
Neither $CS(s)$ nor $CS(s')$ can contain
a term of the form $p+c$ with $1 \le c \le p-1$.
\end{lemma}

\begin{proof}
Suppose on the contrary that $CS(\phi(\alpha))=CS(s)$ or $CS(\phi(\delta))=CS(s')$,
say $CS(s)$, contains a term of the form $p+c$ with $1 \le c \le p-1$.
Then there exist two $2$-cells $D_1$ and $D_2$ in $M$
as illustrated in Figure~\ref{fig.theorem6_1_1}(a),
which follows Convention~\ref{con:figure} below,
such that
\begin{align*}
&\text{(i) $\partial D_1^+ \partial D_2^+$ is a subpath of an outer boundary cycle of $M$;} \\
&\text{(ii) $S(\phi(\partial D_1^+))=(\dots, \ell_1)$,
where $1 \le \ell_1 \le p$;} \\
&\text{(iii) $S(\phi(\partial D_2^+))=(\ell_2, \dots)$,
where $1 \le \ell_2 \le p$; and}\\
&\text{(iv) $S(\phi(\partial D_1^+\partial D_2^+))=(\dots, \ell_1+\ell_2, \dots)$,
where $\ell_1+\ell_2=p+c$.}
\end{align*}
Since $p+c < 2p$, $\ell_1<p$ or $\ell_2<p$.
Here, if $\ell_1<p$ and $\ell_2<p$, then
since both $S(\phi(\partial D_1^-))$ and $S(\phi(\partial D_2^-))$
contain a term $p$ by Lemma~\ref{lem:after_claim1},
we see that $S(\phi(\partial D_1^-\partial D_2^-))$ contains a subsequence $(p, 2p-(\ell_1+\ell_2), p)$
as shown in Figure~\ref{fig.theorem6_1_1}(b).
So $CS(\phi(\delta))=CS(s')$ contains a term $2p-(\ell_1+\ell_2)=p-c<p$,
which contradicts Lemma~\ref{lem:every_term}.
On the other hand,
if $\ell_1<p$ and $\ell_2=p$, then
$S(\phi(\partial D_1^-\partial D_2^-))$ contains a subsequence $(p, p-\ell_1, p)$
as shown in Figure~\ref{fig.theorem6_1_1}(c).
So $CS(\phi(\delta))=CS(s')$ contains a term $p-\ell_1<p$,
again a contradiction to Lemma~\ref{lem:every_term}.
Obviously a similar contradiction is obtained if $\ell_1=p$ and $\ell_2<p$
as shown in Figure~\ref{fig.theorem6_1_1}(d).
\end{proof}

\begin{convention}
\label{con:figure}
{\rm
Recall that $M$ is shaped as in Figure~\ref{fig.layer}(a).
In Figures~\ref{fig.theorem6_1_1} and \ref{fig.theorem6_1_2},
the upper complementary region is regarded as the unbounded region
of $\RR^2-M$. Thus an outer boundary cycle runs the upper boundary from left to right.
Also the change of directions of consecutive arrowheads
represents the change from positive (negative, resp.) words
to negative (positive, resp.) words, and
a dot represents a vertex whose position is clearly identified.
Furthermore, a number such as $\ell_1$, $\ell_2$, $p$, etc
represents the length of the corresponding positive
(or negative) word.
}
\end{convention}

\begin{figure}[h]
\includegraphics{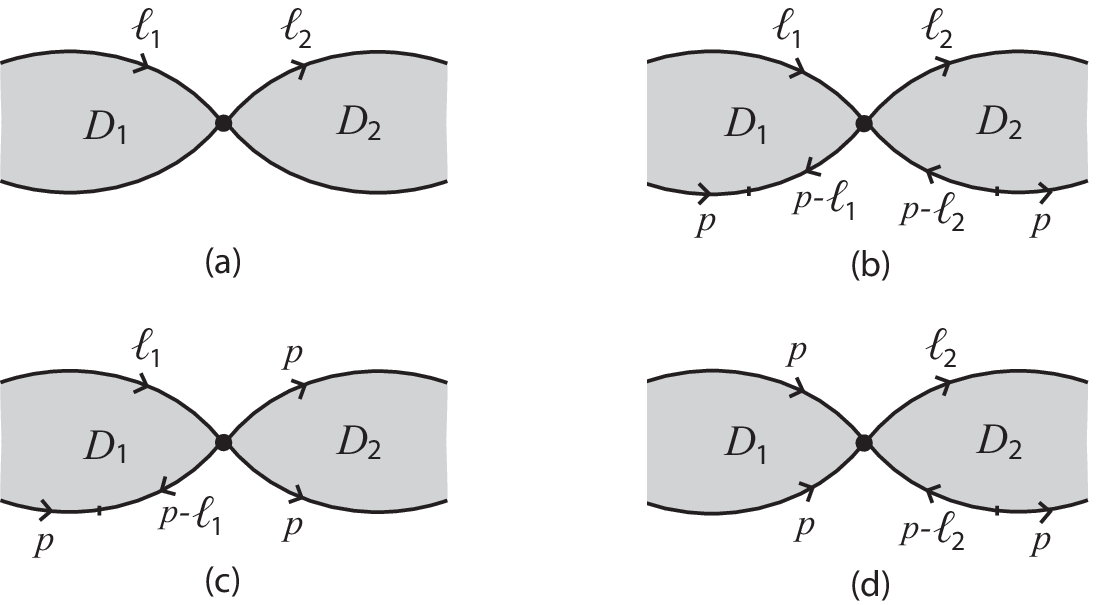}
\caption{Lemma~\ref{lem:ending_cliam1}}
\label{fig.theorem6_1_1}
\end{figure}

\begin{lemma}
\label{lem:ending_cliam2}
Neither $CS(s)$ nor $CS(s')$ can contain a term greater than $2p$.
\end{lemma}

\begin{proof}
Suppose on the contrary that $CS(\phi(\alpha))=CS(s)$ or $CS(\phi(\delta))=CS(s')$,
say $CS(s)$, contains a term greater than $2p$.
Then by Lemma~\ref{lem:after_claim1},
there exist three $2$-cells $D_1$, $D_2$ and $D_3$ in $M$
as illustrated in Figure~\ref{fig.theorem6_1_2}(a),
which follows Convention~\ref{con:figure}, such that
\begin{align*}
&\text{(i) $\partial D_1^+ \partial D_2^+ \partial D_3^+$ is a subpath
of an outer boundary cycle of $M$;} \\
&\text{(ii) $S(\phi(\partial D_1^+))=(\dots, \ell_1)$,
where $1 \le \ell_1 \le p$;} \\
&\text{(iii) $S(\phi(\partial D_2^+))=(p)$;} \\
&\text{(iv) $S(\phi(\partial D_3^+))=(\ell_2, \dots)$,
where $1 \le \ell_2 \le p$; and} \\
&\text{(v) $S(\phi(\partial D_1^+\partial D_2^+ \partial D_3^+))=(\dots, \ell_1+p+\ell_2, \dots)$,
where $\ell_1+p+\ell_2>2p$.}
\end{align*}
Here, if $\ell_1<p$, then
since $S(\phi(\partial D_1^-))$ contains a term $p$ by Lemma~\ref{lem:after_claim1},
and since $S(\phi(\partial D_2^-))=((2n-1) \langle p \rangle)$,
we see that $S(\phi(\partial D_1^-\partial D_2^-))$ contains a subsequence $(p, p-\ell_1, p)$
as shown in Figure~\ref{fig.theorem6_1_2}(b).
So $CS(\phi(\delta))=CS(s')$ contains a term $p-\ell_1<p$,
which contradicts Lemma~\ref{lem:every_term}.
On the other hand,
if $\ell_1=p$, then
$S(\phi(\partial D_1^-\partial D_2^-))$ contains a subsequence $(2p, p, p)$
as shown in Figure~\ref{fig.theorem6_1_2}(c).
So $CS(\phi(\delta))=CS(s')$ contains both a term $p$
and a term $2p+\ell$ with $\ell \ge 0$.
But since $2p+\ell>p+1$, we obtain a contradiction to Lemma~\ref{lem:properties}.
\end{proof}

\begin{figure}[h]
\includegraphics{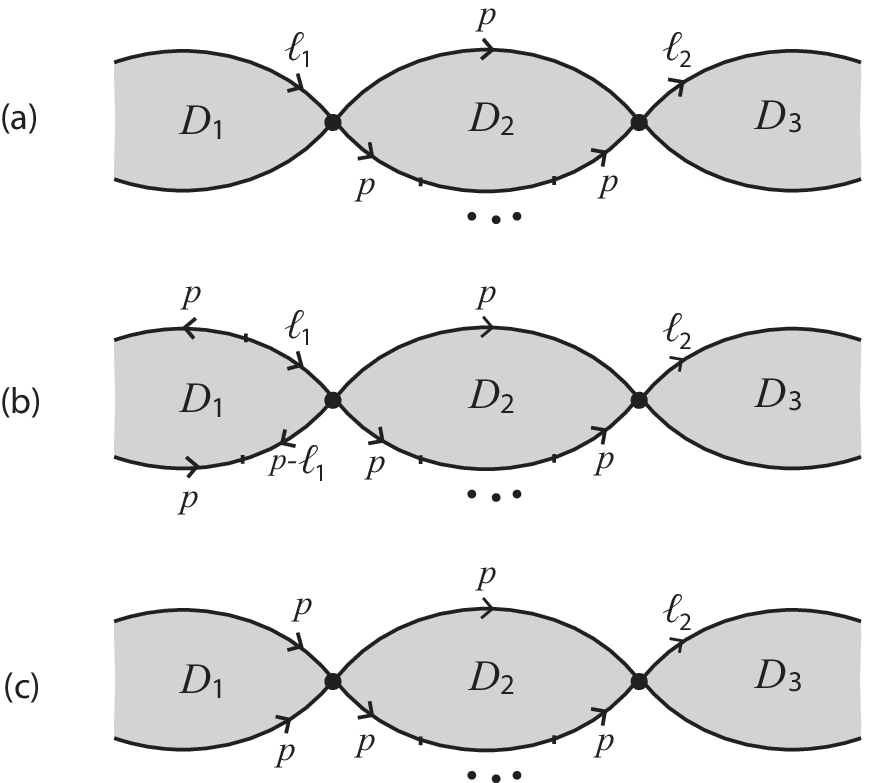}
\caption{Lemma~\ref{lem:ending_cliam2}}
\label{fig.theorem6_1_2}
\end{figure}

By Lemmas~\ref{lem:every_term}, \ref{lem:ending_cliam1} and \ref{lem:ending_cliam2},
the only possibility is that
$CS(s)=CS(s')=\lp 2p, 2p\rp$,
but this is an obvious contradiction, because $s$ and $s'$ are distinct.

The proof of Main Theorem~\ref{thm:conjugacy}(1)
for the case $r=1/p$
is now completed.
\qed

\section{Proof of Main Theorem~\ref{thm:conjugacy}(2) and (3) for the case $r=1/p$}
\label{sec:proof of main theorem(2)}

Main Theorem~\ref{thm:conjugacy}(3) for the case $r=1/p$
can be proved by simply replacing a non-integral rational number $r$ with $1/p$
and by using $S_1=\emptyset$ and $S_2=(p)$ in \cite[Section~4]{lee_sakuma_10}.
So we defer its proof to \cite{lee_sakuma_10}.

It remains to prove Main Theorem~\ref{thm:conjugacy}(2) for the case $r=1/p$.
Suppose on the contrary that there exists a rational number $s$
in $I_1(1/p;n)\cup I_2(1/p;n)$ for which the simple loop $\alpha_s$ is peripheral in $\orbs(1/p;n)$.
Then $u_s$ is conjugate to $a^{\pm t}$ or $b^{\pm t}$ in $\Hecke(1/p;n)$ for some integer $t \ge 1$.
We assume that $u_s$ is conjugate to $a^{\pm t}$ in $\Hecke(1/p;n)$.
(The case when $u_s$ is conjugate to $b^{\pm t}$ in $\Hecke(1/p;n)$ is treated similarly.)
By Lemma~\ref{lem:lyndon_schupp}, there is a reduced nontrivial annular
diagram $M$ over $\Hecke(1/p;n)=\langle a, b \svert u_{1/p}^n \rangle$ with
$(\phi(\alpha)) \equiv (u_s)$ and $(\phi(\delta)) \equiv (a^{\pm t})$,
where $\alpha$ and $\delta$ are, respectively, outer and inner boundary cycles of $M$.
Furthermore, by Corollary~\ref{small_cancellation_condition_2},
$M$ is a $[4,4n]$-map.

Let the outer and inner boundaries of $M$ be denoted by $\sigma$ and $\tau$, respectively.

\begin{lemma}
\label{lem:simple_boundary}
The outer and inner boundaries $\sigma$ and $\tau$ are simple.
\end{lemma}

\begin{proof}
If $\sigma$ is not simple, then the same proof of Lemma~\ref{lem:degree2_vertices}
yields a contradiction.
So suppose that $\tau$ is not simple.
Then there is an extremal disk, say $J$, such that
$J$ is properly contained in $M$ with $\partial J \subset \tau$
and connected to the rest of $M$ by a single vertex.
Then $J$ is a connected and simply-connected $[4,4n]$-map.
Since $(u_{1/p}^n)$ is alternating but $(\phi(\delta)) \equiv (a^{\pm t})$ is not,
there is no vertex in $\partial J$ with $d_J(v)=2$.
But this is a contradiction to Proposition~\ref{prop:key}(2) applied to $J$.
\end{proof}

\begin{lemma}
\label{lem:no_edge}
There is no edge contained in $\sigma \cap \tau$.
\end{lemma}

\begin{proof}
Suppose on the contrary that $\sigma \cap \tau$ contains an edge.
As illustrated in Figure~\ref{fig.island} in Section~\ref{sec:small_cancellation_theory},
there is a submap $J$ of $M$ such that
\begin{enumerate}[\indent \rm (i)]
\item $J$ is bounded by a simple
closed path of the form $\sigma_1 \tau_1$, where $\sigma_1 \subseteq \sigma$
and $\tau_1 \subseteq \tau$;

\item $J$ is connected to the rest of $M$ by two distinct vertices, say $v_1$ and $v_2$,
where $\sigma_1 \cap \tau_1=\{v_1, v_2 \}$
and $v_1$ is an endpoint of an edge contained in $\sigma \cap \tau$.
Note that $d_J(v_1)=d_M(v_1)-1 \ge 2$ and $d_J(v_2)\ge 2$.
\end{enumerate}
Then $J$ is a connected and simply connected $[4,4n]$-map.
Since $(u_{1/p}^n)$ is alternating but $(\phi(\delta)) \equiv (a^{\pm t})$ is not,
there is no vertex in $\tau_1 -\{v_1,v_2\}$ with $d_J(v)=2$.
Also, since both $(u_{1/p}^n)$ and $(\phi(\alpha)) \equiv (u_s)$ are alternating,
there is no vertex in $\sigma_1 -\{v_1,v_2\}$ with $d_J(v)=3$.
So we put
\[
\begin{aligned}
A= &\ \text{\rm the number of vertices $v$ in $\sigma_1 -\{v_1,v_2\}$ such that $d_J(v)=2$;}\\
B_1= &\ \text{\rm the number of vertices $v$ in $\sigma_1-\{v_1,v_2\}$ such that $d_J(v) \ge 4$;}\\
B_2= &\ \text{\rm the number of vertices $v$ in $\tau_1-\{v_1,v_2\}$ such that $d_J(v) \ge 4$;}\\
C= &\ \text{\rm the number of vertices $v$ in $\tau_1 -\{v_1,v_2\}$ such that $d_J(v)=3$.}
\end{aligned}
\]
Since $s \in I_1(1/p;n)\cup I_2(1/p;n)$,
we see by Corollary~\ref{cor:consecutive_vertices}
together with Convention~\ref{convention}(3) that
$4n-2$ degree $2$ vertices cannot occur consecutively on $\sigma_1-\{v_1, v_2\}$,
so that $A \le (4n-3)(B_1+1)$.
We will derive a contradiction to this inequality
using Proposition~\ref{prop:key}(2) applied to $J$.

First if $d_J(v_1)=2$ and $d_J(v_2)=2$, then
\[
\begin{aligned}
A+2 & \ge (4n-3)(B_1+B_2)+(4n-4)\lceil C/2 \rceil+4n \\
& \ge (4n-3)B_1+4n =(4n-3)(B_1+1)+3,
\end{aligned}
\]
so that $A \ge (4n-3)(B_1+1)+1$,
contrary to $A \le (4n-3)(B_1+1)$.
Second either if $d_J(v_1)=2$ and $d_J(v_2)=3$
or if $d_J(v_1)=3$ and $d_J(v_2)=2$,
then
\[
\begin{aligned}
A+1 &\ge (4n-3)(B_1+B_2)+(4n-4)\lceil (C+1)/2 \rceil+4n \\
& \ge (4n-3)B_1+4n =(4n-3)(B_1+1)+3,
\end{aligned}
\]
so that $A \ge (4n-3)(B_1+1)+2$,
contrary to $A \le (4n-3)(B_1+1)$.
Third either if $d_J(v_1)=2$ and $d_J(v_2) \ge 4$
or if $d_J(v_1) \ge 4$ and $d_J(v_2)=2$,
then
\[
\begin{aligned}
A+1 &\ge (4n-3)(B_1+B_2+1)+(4n-4)\lceil C/2 \rceil+4n \\
& \ge (4n-3)(B_1+1)+4n,
\end{aligned}
\]
so that $A \ge (4n-3)(B_1+1)+4n-1$,
contrary to $A \le (4n-3)(B_1+1)$.
Fourth if $d_J(v_1)=3$ and $d_J(v_2)=3$, then
\[
\begin{aligned}
A & \ge (4n-3)(B_1+B_2)+(4n-4)\lceil (C+2)/2 \rceil+4n \\
& \ge (4n-3)B_1+4n =(4n-3)(B_1+1)+3,
\end{aligned}
\]
contrary to $A \le (4n-3)(B_1+1)$.
Fifth either if $d_J(v_1)=3$ and $d_J(v_2) \ge 4$
or if $d_J(v_1) \ge 4$ and $d_J(v_2)=3$, then
\[
\begin{aligned}
A & \ge (4n-3)(B_1+B_2+1)+(4n-4)\lceil (C+1)/2 \rceil+4n \\
& =(4n-3)(B_1+1)+4n,
\end{aligned}
\]
contrary to $A \le (4n-3)(B_1+1)$.
Finally if $d_J(v_1)\ge 4$ and $d_J(v_2) \ge 4$, then
\[
\begin{aligned}
A &\ge (4n-3)(B_1+B_2+2)+(4n-4)\lceil C/2 \rceil+4n \\
& =(4n-3)(B_1+2)+4n,
\end{aligned}
\]
contrary to $A \le (4n-3)(B_1+1)$.
\end{proof}

At this point, we newly let $A$, $B$ and $C$ be the numbers defined for $M$
as in Proposition~\ref{prop:key}. Then by Proposition~\ref{prop:key}(2) applied to $M$,
we have $A \ge (4n-3)B+(4n-4)\lceil C/2 \rceil$.
Furthermore, we newly put
\[
\begin{aligned}
A_1= &\ \text{\rm the number of vertices $v$ in $\sigma$ such that $d_M(v)=2$;} \\
B_1= &\ \text{\rm the number of vertices $v$ in $\sigma$ such that $d_M(v) \ge 4$.}
\end{aligned}
\]

\medskip
\noindent {\bf Claim.} {\it $A_1 > (4n-3)B_1$.}

\begin{proof}[Proof of Claim]
First suppose that $\sigma \cap \tau = \emptyset$.
Since $(u_{1/p}^n)$ is alternating but $(\phi(\delta)) \equiv (a^{\pm t})$ is not,
there is no vertex in $\tau$ with degree $2$,
and therefore $A=A_1$.
Clearly there is at least one vertex in $\tau$, say $v$.
Since $d_M(v) \ge 3$,
we have $(4n-3)B+(4n-4)\lceil C/2 \rceil > (4n-3)B_1$.
Here, since $A=A_1$, we have
$A_1 > (4n-3)B_1$, as desired.

Next suppose that $\sigma \cap \tau \neq \emptyset$.
By Lemma~\ref{lem:no_edge}, $\sigma \cap \tau$ consists of
finitely many vertices in $M$.
Then for every $v \in \sigma \cap \tau$, clearly $d_M(v) \ge 4$.
Furthermore, since $(\phi(\alpha)) \equiv (u_s)$ is
alternating while $(\phi(\delta)) \equiv (a^{\pm t})$ is not,
we see that $d_M(v)$ is an odd number.
Since $d_M(v) \ge 4$, this implies $d_M(v) \ge 5$
and therefore $M$ does not satisfy the condition in Remark~\ref{rem:big_degree_vertex}(2).
Hence $A > (4n-3)B+(4n-4)\lceil C/2 \rceil$ by the remark.
Clearly $(4n-3)B+(4n-4)\lceil C/2 \rceil \ge (4n-3)B_1$.
Here, since $A=A_1$ by reasoning as above,
we have $A_1 > (4n-3)B_1$, as desired.
\end{proof}

The above claim implies that $\sigma$ contains
$4n-2$ consecutive degree $2$ vertices.
But then the cyclic word $(\phi(\alpha)) \equiv (u_s)$ contains a subword
$w$ of $(u_{1/p}^{\pm n})$ which is a product of $4n-1$ pieces
but is not a product of less than $4n-1$ pieces.
This contradiction to Corollary~\ref{cor:consecutive_vertices}
completes the proof of Main Theorem~\ref{thm:conjugacy}(2) for the case $r=1/p$.
\qed

\bibstyle{plain}


\bigskip
\begin{thebibliography}{3}

\bibitem{BMP}
M.\ Boileau, S,\ Maillot, Sylvain and J.\ Porti,
{\em Three-dimensional orbifolds and their geometric structures},
Panoramas et Synth\'eses, {\bf 15},
Soci\'et\'e Math\'ematique de France, Paris, 2003.

\bibitem{Boileau-Porti}
M.\ Boileau and J.\ Porti,
{\em Geometrization of $3$-orbifolds of cyclic type},
Appendix A by Michael Heusener and Porti,
Ast\'erisque No. 272 (2001).

\bibitem{Gordon} C. Gordon,
{\em Problems},
Workshop on Heegaard Splittings, 401--411, Geom.
Topol. Monogr. {\bf 12}, Geom. Topol. Publ., Coventry, 2007.


\bibitem{Hecke}
E.\ Hecke,
{\em \"Uber die Bestimung Dirichletscher Reihen
durch ihre Funktionalgleichung},
Math. Ann. {\bf 112} (1936), 664--699.

\bibitem{lee_sakuma_0}
D.\ Lee and M.\ Sakuma,
{\em Simple loops on $2$-bridge spheres in $2$-bridge link complements},
Electron. Res. Announc. Math. Sci. {\bf 18} (2011), 97--111.

\bibitem{lee_sakuma}
D.\ Lee and M.\ Sakuma,
{\em Epimorphisms between $2$-bridge link groups:
homotopically trivial simple loops on $2$-bridge spheres},
Proc. London Math. Soc. {\bf 104} (2012), 359--386.

\bibitem{lee_sakuma_8}
D.\ Lee and M.\ Sakuma,
{\em Simple loops on $2$-bridge spheres in Heckoid orbifolds for $2$-bridge links},
Electron. Res. Announc. Math. Sci. {\bf 19} (2012), 97--111.

\bibitem{lee_sakuma_5}
D.\ Lee and M.\ Sakuma,
{\em A variation of McShane's identity for $2$-bridge links},
Geom. Topol. {\bf 17} (2013), 2061--2101.

\bibitem{lee_sakuma_6}
D.\ Lee and M.\ Sakuma,
{\em Epimorphisms from $2$-bridge link groups onto Heckoid groups {\rm (I)}},
Hiroshima Math. J. {\bf 43} (2013), 239--264.

\bibitem{lee_sakuma_7}
D.\ Lee and M.\ Sakuma,
{\em Epimorphisms from $2$-bridge link groups onto Heckoid groups {\rm (II)}},
Hiroshima Math. J. {\bf 43} (2013), 265--284.

\bibitem{lee_sakuma_2}
D.\ Lee and M.\ Sakuma,
{\em Homotopically equivalent simple loops
on $2$-bridge spheres in $2$-bridge link complements {\rm (I)}},
to appear in Geom. Dedicata, arXiv:1010.2232.

\bibitem{lee_sakuma_3}
D.\ Lee and M.\ Sakuma,
{\em Homotopically equivalent simple loops
on $2$-bridge spheres in $2$-bridge link complements {\rm (II)}},
to appear in Geom. Dedicata, arXiv:1103.0856.

\bibitem{lee_sakuma_4}
D.\ Lee and M.\ Sakuma,
{\em Homotopically equivalent simple loops
on $2$-bridge spheres in $2$-bridge link complements {\rm (III)}},
to appear in Geom. Dedicata, arXiv:1111.3562.

\bibitem{lee_sakuma_10}
D.\ Lee and M.\ Sakuma,
{\em Homotopically equivalent simple loops
on $2$-bridge spheres in Heckoid orbifolds for $2$-bridge links {\rm (II)}},
arXiv:1402.6873.


\bibitem{lyndon_schupp}
R. C.\ Lyndon and P. E.\ Schupp,
{\em Combinatorial group theory},
Springer-Verlag, Berlin, 1977.

\bibitem{Ohshika-Sakuma}
K.\ Ohshika and M.\ Sakuma,
{\em Subgroups of mapping class groups related to Heegaard splittings and bridge decompositions},
arXiv:1308.0888.

\bibitem{Ohtsuki-Riley-Sakuma}
T.\ Ohtsuki, R.\ Riley, and M.\ Sakuma,
{\em Epimorphisms between $2$-bridge link groups},
Geom. Topol. Monogr. {\bf 14} (2008), 417--450.


\bibitem{Riley}
R.\ Riley,
{\em Parabolic representations of knot groups, {\rm I}},
Proc. London Math. Soc. {\bf 24} (1972), 217--242.

\bibitem{Riley2}
R.\ Riley,
{\em Algebra for Heckoid groups},
Trans. Amer. Math. Soc. {\bf 334} (1992), 389--409.

\end{thebibliography}
\end{document}